\documentclass[11pt,twoside]{article}

\usepackage{mathrsfs,amsfonts,amsmath,amssymb}
\usepackage{latexsym}
\newtheorem{satz}{Theorem}
\newtheorem{proposition}[satz]{Proposition}
\newtheorem{theorem}[satz]{Theorem}
\newtheorem{lemma}[satz]{Lemma}
\newtheorem{definition}[satz]{Definition}
\newtheorem{corollary}[satz]{Corollary}
\newtheorem{remark}[satz]{Remark}
\newtheorem{example}[satz]{Example}

\def\no{\noindent}

\def\eps{\varepsilon}
\def\_phi{\varphi}

\def\a{\alpha}

\def\la{\lambda}

\def\F{{\mathbb F}}
\def\L{\Lambda}
\def\m{\times}
\def\t{\tilde}

\def\ov{\overline}

\def\C{{\mathbb C}}
\def\R{{\mathbb R}}
\def\E{\mathsf {E}}
\def\T{{\mathbb T}}

\def\Z_N{{\mathbb Z}_N}
\def\Z{{\mathbb Z}}
\def\N{{\mathbb N}}

\def\Gr{{\mathbf G}}

\def\D{{\mathbb D}}

\def\oT{{\rm T}}

\def\G{\Gamma}
\def\FF{\widehat}
\def\c{\circ}
\def\D{\Delta}
\def\Cf{{\mathcal C}}

\def\gs{\geqslant}
\def\T{\mathsf {T}}

\def\M{\mathbf {M}}

\author{Shkredov I.D.}
\title{ Some new results on higher energies
\footnote{
This work was supported by grant RFFI NN
06-01-00383, 11-01-00759, Russian Government project 11.G34.31.0053,
Federal Program "Scientific and scientific--pedagogical staff of innovative Russia" 2009--2013
and
grant Leading Scientific Schools N 2519.2012.1.}
}
\date{}
\begin{document}
\maketitle

\begin{center}
 Annotation.
\end{center}

{\it \small
    In the paper we develop the method of higher energies.
    New upper bounds for the additive energies of convex sets, sets $A$ with small $|AA|$
    and
    $|A(A+1)|$ are obtained.
    We prove new structural results, including higher sumsets,
    and develop the notion of dual popular difference sets.
}
\\

\section{Introduction}
\label{sec:introduction}

The method of higher energies (or, in other words, the method of higher moments of convolutions)
was introduced in \cite{ss}, was developed in \cite{ss_E_k,sv}
and found a series of applications in \cite{A(A+1),KR,Li,Li2,ss2,s_ineq,s_heilbronn}.
In the paper we obtain some new results in the direction, using
so--called operator (or eigenvalues) method from \cite{s,s_ineq},
which we recall in section \ref{sec:eigenvalues}.

Our main results are contained in three sections \ref{sec:applications}--\ref{sec:applications3}.
In section \ref{sec:applications} we apply the eigenvalues method to obtain new
upper bounds for the additive energy of some families of sets.
Let us formulate just one result
in the direction
which concerns convex subset (that is the image of a convex map) of $\R$.

\begin{theorem}
    Let $A \subseteq \R$ be a convex set.
    Then
    \begin{equation}\label{f:convex_energy_pred}
        \E (A) \ll |A|^{\frac{32}{13}} \log^{\frac{71}{65}} |A| \,.
    \end{equation}
\label{t:convex_energy_pred}
\end{theorem}

Here $\E(A)$, so--called, the {\it additive energy} of our set $A$,
which equals
the number of solution of the equation $a_1-a_2=a_3-a_4$, where $a_1,a_2,a_3,a_4\in A$.
Constant $\frac{5}{2}$ instead of $\frac{32}{13}$ was obtained in \cite{ikrt},
and after that
this constant was improved to $\frac{89}{36}$ in \cite{s_ineq}, using the eigenvalues method again.

The next section contains so--called structural results.
The most important example of such statements in additive combinatorics is, of course,
beautiful
Freiman's theorem on sets with small doubling, or, in other words, sets having small sumset,
which gives a full description of the family of sets (see \cite{tv}).
In the paper by structural results we mean another thing, namely,
having some condition on a set (basically, on higher energies of this set)
we
wish
to find some subsets of the set having small doubling or large additive energy.
As an example of such type of results, we recall  a strong structural theorem from \cite{BK_struct}.

\begin{theorem}
Let $A \subseteq \Gr$ be a symmetric set, $\tau_0$, $\sigma_0$ be nonnegative real numbers
and $A$ has the property that for any $A_*\subseteq A$, $|A_*| \gg |A|$ the following holds
$\E(A_*) \gg \E (A) = |A|^{2+\tau_0}$.
Suppose that $\T_4 (A) \ll |A|^{4+3\tau_0+\sigma_0}$.
Then there exists a function $f_{\tau_0} : (0,1) \to (0,\infty)$ with $f_{\tau_0} (\eta) \to 0$ as $\eta \to 0$
and a number $\a\ge 0$ such that there are sets $X_j,H_j\subseteq \Gr$, $B_j\subseteq A$,
$j\in [|A|^{\a-f_{\tau_0} (\sigma_0)}]$ with
$$
    |H_j| \ll |A|^{\tau_0+\a+ f_{\tau_0} (\sigma_0)} \,,\quad \quad |X_j| \ll |A|^{1-\tau_0-2\a+ f_{\tau_0} (\sigma_0)} \,,
$$
$$
    |H_j-H_j| \ll |H_j|^{1+f_{\tau_0} (\sigma_0)} \,,
$$
$$
    |(X_j+H_j) \cap B_j| \gg |A|^{1-\a-f_{\tau_0} (\sigma_0)} \,,
$$
and $B_i \cap B_j = \emptyset$ for all $i\neq j$.
\label{t:BK_structural}
\end{theorem}

Here $\T_4 (A)$ is the number solution of the equation $a_1+a_2+a_3+a_4=a'_1+a'_2+a'_3+a'_4$,
$a_1,a_2,a_3,a_4,a'_1,a'_2,a'_3,a'_4 \in A$.

In the paper we need in a generalization of the
notion of the additive energy of a set.
Namely,
for all real $s\ge 1$ put
\begin{equation}\label{f:E_k_introduction}
    \E_s (A) = \sum_x |A\cap (A-x)|^s \,,
\end{equation}
Quantities $\E_s (A)$ from (\ref{f:E_k_introduction}) are exactly
that
we call higher energies.

\bigskip

Let us formulate our two main structural results.
The weaker forms of the first one were proved in \cite{ss_E_k} and \cite{s_ineq}.
From some point of view
these type of statements can be  called an optimal version of
Balog--Szemer\'{e}di--Gowers theorem, see \cite{ss_E_k}.

\begin{theorem}
    Let $A\subseteq \Gr$ be a set, $\E (A) = |A|^{3}/K^{}$,
    and $\E_3 (A) = M|A|^4 / K^2$.
    Then there is a set $A' \subseteq A$
    such that
    \begin{equation}\label{f:E_3_size_pred}
        |A'| \gg  M^{-10} \log^{-15} M \cdot |A|  \,,
    \end{equation}
    and
    \begin{equation}\label{f:E_3_doubling_pred}
        |nA'-mA'| \ll (M^{9} \log^{14} M)^{6(n+m)} K |A'|
    \end{equation}
    for every $n,m\in \N$.
\label{t:E_2_E_3_pred}
\end{theorem}

Interestingly, that the generality of Theorem \ref{t:E_2_E_3_pred} allows
us to prove a "non--trivial"\,
version
of Theorem \ref{t:convex_energy_pred},
that is a bound of the form $\E(A) \ll |A|^{5/2-\eps_0}$, where $\eps_0 >0$ is an absolute constant.
A similar
proof takes place in the case of multiplicative subgroups of $\Z/p\Z$,
$p$ is a prime number (see Remark \ref{r:worker-peasant}).

\bigskip

The second structural result is the following.

\begin{theorem}
    Let $A\subseteq \Gr$ be a set, $\E_{3/2} (A) = |A|^{5/2}/K^{1/2}$, and $\T_4 (A) = M|A|^7 / K^3$.
    Then there is a set $A' \subseteq A$
    such that
    \begin{equation}\label{f:E_4_size_pred}
        |A'| \gg  \frac{|A|}{M K}  \,,
    \end{equation}
    and
    \begin{equation}\label{f:E_4_energy_pred}
        \E(A') \gg \frac{|A'|^3}{M} \,.
    \end{equation}
\label{t:E_3/2_T_4_introduction}
\end{theorem}

It is easy to see that Theorem \ref{t:E_3/2_T_4_introduction} is tight
(see e.g. Remark \ref{r:L+H}).
The assumption of our result is stronger than the assumption of Theorem \ref{t:BK_structural}.

Popular difference sets are very simple and important objects in additive combinatorics
(see e.g. \cite{Gow_4,Gow_m,tv}).
In section \ref{sec:applications3} we develop an idea of Bateman--Katz from \cite{BK_AP3,BK_struct}
that every popular set has a companion, which we call a dual popular set.
As an application, our method allows find a nontrivial relation between $\E (A)$ and $\E_s (A)$, $s\in [1,2]$,
see Theorem \ref{t:dual_bounds} or Corollary \ref{c:connected}.
It is interesting that for $s>2$ there is no such connection at all.

Note, finally, that the arguments of the paper are elementary in the sense
that they do not use Fourier transform.

The author is grateful to Tomasz Schoen, Sergey Konyagin,
 and Misha Rudnev for fruitful discussions and explanations.

\section{Definitions}
\label{sec:definitions}

Let $\Gr$ be an abelian group.
If $\Gr$ is finite then denote by $N$ the cardinality of $\Gr$.
We define two types of convolutions on $\Gr$
$$
    (f*g) (x) := \sum_{y\in \Gr} f(y) g(x-y) \quad \mbox{ and } \quad (f\circ g) (x) := \sum_{y\in \Gr} f(y) g(y+x) =
        (f*g^c) (-x) \,,
$$
where for a function $f:\Gr \to \mathbb{C}$ we put $f^c (x):= f(-x)$.
 Clearly,  $(f*g) (x) = (g*f) (x)$ and $(f\c g)(x) = (g \c f) (-x)$, $x\in \Gr$.
 The $k$--fold convolution, $k\in \N$  we denote by $*_k$,
 so $*_k := *(*_{k-1})$.

We use in the paper  the same letter to denote a set
$S\subseteq \Gr$ and its characteristic function $S:\Gr\rightarrow \{0,1\}.$
Write $\E(A,B)$ for the {\it additive energy} of two sets $A,B \subseteq \Gr$
(see e.g. \cite{tv}), that is
$$
    \E(A,B) = |\{ a_1+b_1 = a_2+b_2 ~:~ a_1,a_2 \in A,\, b_1,b_2 \in B \}| \,.
$$
If $A=B$ we simply write $\E(A)$ instead of $\E(A,A).$
Clearly,
\begin{equation}\label{f:energy_convolution}
    \E(A,B) = \sum_x (A*B) (x)^2 = \sum_x (A \circ B) (x)^2 = \sum_x (A \circ A) (x) (B \circ B) (x)
    \,.
\end{equation}
Let
$$
   \T_k (A) := \sum_x (A*_{k-1} A)^2 (x)
    =
   | \{ a_1 + \dots + a_k = a'_1 + \dots + a'_k  ~:~ a_1, \dots, a_k, a'_1,\dots,a'_k \in A \} | \,.
$$
Let also
$$
    \sigma_k (A) := (A*_{k-1} A)(0)=| \{ a_1 + \dots + a_k = 0 ~:~ a_1, \dots, a_k \in A \} | \,.
$$
Notice that for a symmetric set $A$ that is $A=-A$ one has $\sigma_2 (A) = |A|$ and $\sigma_{2k} (A) = \T_k (A)$.
If $\psi : \Gr \to \C$ is a function then we write
$$
    \sigma_\psi (A) = \sigma(\psi,A) := \sum_{x} \psi (x) (A\c A) (x) \,.
$$
So, if $P\subseteq \Gr$ is another set then put $\sigma_P (A) := \sum_{x\in P} (A\c A) (x)$.
Similarly, write $\E_P (A) := \sum_{x\in P} (A\c A)^2 (x)$.

 For a sequence $s=(s_1,\dots, s_{k-1})$ put
$A^B_s = B \cap (A-s_1)\dots \cap (A-s_{k-1}).$
If $B=A$ then write $A_s$ for $A^A_s$.
Let
\begin{equation}\label{f:E_k_preliminalies}
    \E_k(A)=\sum_{x\in \Gr} (A\c A)(x)^k = \sum_{s_1,\dots,s_{k-1} \in \Gr} |A_s|^2
\end{equation}
and
\begin{equation}\label{f:E_k_preliminalies_B}
\E_k(A,B)=\sum_{x\in \Gr} (A\c A)(x) (B\c B)(x)^{k-1} = \sum_{s_1,\dots,s_{k-1} \in \Gr} |B^A_s|^2
\end{equation}
be the higher energies of $A$ and $B$.
The second formulas in (\ref{f:E_k_preliminalies}), (\ref{f:E_k_preliminalies_B})
can be considered as the definitions of $\E_k(A)$, $\E_k(A,B)$ for non integer $k$, $k\ge 1$.
As above for a set $P\subseteq \Gr$ write $\E^P_k (A) := \sum_{s\in P} |A_s|^k$
and for a set $\mathcal{P} \subseteq \Gr^{k-1}$ put
$$
    \E^{\mathcal{P}}_k (A) := \sum_{(s_1,\dots,s_{k-1}) \in \mathcal{P}} |A_s|^2 \,.
$$


Clearly,
\begin{eqnarray}\label{f:energy-B^k-Delta}
\E_{k+1}(A,B)&=&\sum_x(A\c A)(x)(B\c B)(x)^{k}\nonumber \\
&=&\sum_{x_1,\dots, x_{k-1}}\Big (\sum_y A(y)B(y+x_1)\dots
B(y+x_{k})\Big )^2 =\E(\Delta_k (A),B^{k}) \,,
 \end{eqnarray}
where
$$
    \Delta (A) = \Delta_k (A) := \{ (a,a, \dots, a)\in A^k \}\,.
$$
We also put $\Delta(x) = \Delta (\{ x \})$, $x\in \Gr$.


Quantities $\E_k (A,B)$ can be written in terms of generalized convolutions.

\begin{definition}
   Let $k\ge 2$ be a positive number, and $f_0,\dots,f_{k-1} : \Gr \to \C$ be functions.
Write $F$ for the vector $(f_0,\dots,f_{k-1})$ and $x$ for vector $(x_1,\dots,x_{k-1})$.
Denote by
$${\mathcal C}_k (f_0,\dots,f_{k-1}) (x_1,\dots, x_{k-1})$$
the function
$$
    \Cf_k(F) (x) =  \Cf_k (f_0,\dots,f_{k-1}) (x_1,\dots, x_{k-1}) = \sum_z f_0 (z) f_1 (z+x_1) \dots f_{k-1} (z+x_{k-1}) \,.
$$
Thus, $\Cf_2 (f_1,f_2) (x) = (f_1 \circ f_2) (x)$.
If $f_1=\dots=f_k=f$ then write
$\Cf_k (f) (x_1,\dots, x_{k-1})$ for $\Cf_k (f_1,\dots,f_{k}) (x_1,\dots, x_{k-1})$.
\end{definition}

In particular, $(\Delta_k (B) \c A^k) (x_1,\dots,x_k) = \Cf_{k+1} (B,A,\dots,A) (x_1,\dots,x_k)$, $k\ge 1$.

\bigskip

The following lemma from \cite{s_ineq} contains all basic properties of quantities $\Cf_k (f_0,\dots,f_{k-1})$.

\begin{lemma}
    For any functions the following holds
$$
    \sum_{x_1,\dots, x_{l-1}} \Cf_l (f_0,\dots,f_{l-1}) (x_1,\dots, x_{l-1})\, \Cf_l (g_0,\dots,g_{l-1}) (x_1,\dots, x_{l-1})
        =
$$
\begin{equation}\label{f:scalar_C}
        =
        \sum_z (f_0 \circ g_0) (z) \dots (f_{l-1} \circ g_{l-1}) (z) \quad \quad  \mbox{\bf (scalar product), }
\end{equation}
moreover
$$
    \sum_{x_1,\dots, x_{l-1}} \Cf_l (f_0) (x_1,\dots, x_{l-1}) \dots  \Cf_l (f_{k-1}) (x_1,\dots, x_{l-1})
        =
$$
\begin{equation}\label{f:gen_C}
        =
            \sum_{y_1,\dots,y_{k-1}} \Cf^l_k (f_0,\dots,f_{k-1}) (y_1,\dots,y_{k-1})
                \quad \quad  \mbox{\bf (multi--scalar product), }
\end{equation}
    and
$$
    \sum_{x_1,\dots, x_{l-1}} \Cf_l (f_0) (x_1,\dots, x_{l-1})\, (\Cf_l (f_1) \circ \dots \circ \Cf_l (f_{k-1})) (x_1,\dots, x_{l-1})
        =
$$
\begin{equation}\label{f:conv_C}
        =
            \sum_{z} (f_0 \circ \dots \circ f_{k-1})^l (z)
                \quad \quad  \mbox{\bf (} \sigma_{k} \quad \mbox{\bf for } \quad \Cf_l \mbox{\bf )}  \,.
\end{equation}
\label{l:commutative_C}
\end{lemma}

Generalizing the notion of $\sigma_P (A)$, $P\subseteq \Gr$ we define
for a set $\mathcal{P} \subseteq \Gr^{k-1}$, $k\ge 2$ the quantity
$$
    \sigma_{\mathcal{P}} (A) := \sum_{(s_1,\dots,s_{k-1}) \in \mathcal{P}} \Cf_k (A) (s_1,\dots,s_{k-1}) \,.
$$

\bigskip

Let $f_1,\dots,f_t : \Gr \to \C$ be functions.
Tensor power of the functions is defined as
$$(f_1 \otimes f_2 \otimes \dots \otimes f_t) (x_1,\dots,x_t) = f_1 (x_1) f_2 (x_2) \dots f_t (x_t) \,.$$
Tensor power of a single function $f$ is denoted by
$(f^\otimes) (x_1,\dots,x_t) = \prod_{j=1}^t f(x_j)$.
So, with some abuse of the notation we
do not write
the number $t$ in the definition of tensor power.
It is easy to see that
\begin{equation}\label{f:tensor_convolutions}
    (g \c f)^\otimes = (g^\otimes \c f^\otimes)
        \quad \mbox{ and } \quad
            (g * f)^\otimes = (g^\otimes * f^\otimes)
\end{equation}
and moreover
\begin{equation}\label{f:tensor_convolutions_C_k}
    \Cf_{k} (f^\otimes_0, \dots, f^\otimes_{k-1}) = \Cf^\otimes_k (f_0,\dots,f_{k-1}) \,.
\end{equation}

\bigskip

For a positive integer $n,$ we set $[n]=\{1,\ldots,n\}$.
All logarithms used in the paper are to base $2.$
By  $\ll$ and $\gg$ we denote the usual Vinogradov's symbols.
We write $\ll_M$ and $\gg_M$ if there is a dependence on a constant $M$.


\section{Preliminaries}
\label{sec:preliminaries}

At the beginning of the section we collect some results
about
matrices.
We need in a lemma on singular decomposition (see e.g. \cite{ss_E_k}).

\begin{lemma}
    Let $n,m$ be two positive integers, $n\le m$, and let $X,Y$ be  sets of cardinalities $n$ and $m$,
    respectively.
    Let also $\M=\M(x,y)\,$, $x\in X$, $y\in Y$  be $n\m m$
    complex (real) matrix.
    Then there are complex (real) functions $u_j$ defined on $X$,
    $v_j$ defined on $Y$, and non--negative numbers $\la_j$
    such that
    \begin{equation}\label{f:M_singular_decomposition_basic}
        \M (x,y) = \sum_{j=1}^n \la_j u_j (x) \ov{v_j (y)} \,,
    \end{equation}
    where $\{ u_j \}$, $j\in [n]$,  and $\{ v_j \}$, $j\in [n]$ form two orthonormal sequences,
    and
    \begin{equation}\label{}
        \la_1 = \max_{w\neq 0} \frac{\| \M w\|_2}{\|w\|_2}\,, \quad \la_2 = \max_{w\neq 0,\, w\perp u_1} \frac{\| \M w\|_2}{\|w\|_2}
            \,, \dots \,, \la_n = \max_{w\neq 0,\, w\perp u_1,\, \dots,\, w\perp u_{n-1}} \frac{\| \M w\|_2}{\|w\|_2} \,.
    \end{equation}
    Further\\
    $\bullet$ $\M u_j = \la_j v_j$, $j\in [n]$. \\
    $\bullet$ The numbers $\la^2_j$ and the vectors $u_j$ are all eigenvalues and eigenvectors of the matrix $\M^* \M$. \\
    $\bullet$ The numbers $\la^2_j$ and the vectors $v_j$ form $n$ eigenvalues and eigenvectors of the matrix $\M \M^*$.
        Another $(m-n)$ eigenvalues of $\M\M^*$ equal zero. \\
    $\bullet$ We have $\sum_{j=1}^n \la^2_j = \sum_{x,y} |\M^2 (x,y)|$, and
                \begin{equation}\label{f:ractangular_norm_of_M}
                    \sum_{j=1}^n \la^4_j = \sum_{x,x'} \Big| \sum_y \M(x,y) \ov{\M(x',y)} \Big|^2 \,.
                \end{equation}
\label{l:singular_decomposition}
\end{lemma}

We will call functions $\{ u_j \}$, $j\in [n]$,  and $\{ v_j \}$, $j\in [n]$
as singularfunctions.

\bigskip

Now we recall the well--known theorem of Perron--Frobenius about
the dominate eigenvalue
and correspondent nonnegative eigenvector of nonnegative matrices (see e.g. \cite{Horn-Johnson}, chapter 8).
By $\rho (M)$ denote the spectral radius of a square matrix $M$.

\begin{theorem}
    Let $M$ be a real
    square
    matrix  with nonnegative entries.
    Then
    eigenvalue $\rho(M)$
    corresponds to a
    nonnegative eigenvector.
    Conversely, if $M$ has a strictly positive eigenvector then this eigenvector corresponds to $\rho(M)$.
\label{t:Perron-Frobenius}
\end{theorem}

\bigskip

Also we need in a particular convex property of eigenvalues (see e.g. \cite{Horn-Johnson}).

\begin{lemma}
    Let $M$ be a normal $(n\times n)$ matrix
    with eigenvalues $\mu_1,\dots,\mu_n$
    and $f$ be a convex function on $n$ complex variables.
    Then
$$
    \max_{ x_1, \dots, x_n } f( \langle M x_1, x_1 \rangle, \dots, \langle M x_n, x_n \rangle)
        =
            \max_{i_1,\dots, i_n} f(\mu_{i_1},\dots,\mu_{i_n}) \,,
$$
where the left maximum is taken over all systems of orthonormal vectors $x_1,\dots,x_n$,
and the right maximum is taken over all permutations of $\{1,2,\dots,n\}$.
\label{l:convex_eigenvalues}
\end{lemma}

\bigskip

Now recall some combinatorial results.

The
first
lemma is a special case of  Lemma 2.8 from \cite{sv}.

\begin{lemma}\label{l:E_k-identity}
 Let $A$ be a subset of an abelian group.
Then for every $k,l\in \N$
$$\sum_{s,t:\atop \|s\|=k-1,\, \|t\|=l-1} \E(A_s,A_t)=\E_{k+l}(A) \,,$$
where $\|x\|$ denote the number of components of vector $x$.
\end{lemma}

\bigskip

We need in the Balog--Szemer\'{e}di--Gowers theorem,
see \cite{tv} section 2.5 and also \cite{schoen_BSzG}.


\begin{theorem}
\label{t:BSzG}
Let $\a \in (0,1]$ be a real number,
$A$ and $B$ be finite sets of an abelian group, and $|A|\gs |B|$.
If $\E(A,B)=\a |A|^3,$ then there exist sets $A'\subseteq A$ and
$B'\subseteq B$ such that $|A'|\gg \a |A|$, $|B'| \gg \a |B|$ and
$$
    |A'+B'|\ll \a^{-5}|A| \,.
$$
\end{theorem}



\bigskip

Recall that a set $A=\{a_1,\dots, a_n\} \subseteq \R$ is called {\it convex}
if $a_i-a_{i-1}<a_{i+1}-a_i$ for every $2\le i\le n-1.$
We need in a lemma, see e.g. \cite{ss2}, \cite{ikrt} or \cite{Li}.

\begin{lemma}
    Let $A$ be a convex set, $A'\subseteq A$, and $B$ be an arbitrary set.
    Then
    \begin{equation}\label{f:convex_A'}
        |A'+B|\gg |A'|^{3/2}|B|^{1/2}|A|^{-1/2} \,.
    \end{equation}
    Arranging $(A*_{k-1} A) (x_1) \ge (A*_{k-1} A) (x_2) \ge \dots $, we have
    \begin{equation}\label{f:E_3_gen_2-}
        (A*_{k-1} A) (x_j) \ll_k |A|^{k-\frac{4}{3} (1-2^{-k})} j^{-\frac{1}{3}} \,.
    \end{equation}
    In particular
    $$
        \E_3 (A) \ll |A|^3 \log |A| \,,
    $$
    and
    $$
        \E(A,B) \ll |A| |B|^{\frac{3}{2}} \,.
    $$
\label{l:E_3_convex}
\end{lemma}

As was realized by Li \cite{Li} (see also \cite{ss_E_k}) that subsets $A$ of real numbers
with small multiplicative doubling looks like convex sets.
More precisely, the following lemma from \cite{ss_E_k} holds.

\begin{lemma}
    Let $A,B \subseteq \R$ be finite sets and let $|AA| = M |A|$.
    Then arranging $(A\c B) (x_1) \ge (A\c B) (x_2) \ge \dots $, we have
$$
    (A\c B) (x_j) \ll (M \log M)^{2/3} |A|^{1/3} |B|^{2/3} j^{-1/3} \,.
$$
In particular
$$
    \E (A,B) \ll M \log M |A| |B|^{3/2} \,.
$$
\label{l:arranging_gen}
\end{lemma}

\section{Operators}
\label{sec:eigenvalues}


In the section we describe the family of operators (finite matrices).
Using such operators, we obtain a series of inequalities from papers \cite{Li}, \cite{ss}, \cite{ss2},
\cite{s_ineq} and others by
uniform
way.
Also we prove several lemmas which we need in the next sections.
Our notations here differ from the paper \cite{s} and do not use Fourier transform.

Let $g : \Gr \to \C$ be a function, and $A,B\subseteq \Gr$ be two finite sets.
Suppose that $|B| \le |A|$.
By $\oT^{g}_{A,B}$ denote the rectangular matrix
\begin{equation}\label{def:operator1'}
    \oT^{g}_{A,B} (x,y) = g(x-y) A(x) B(y) \,,
\end{equation}
and by $\t{\oT}^{g}_{A,B} (x,y)$ denote the another rectangular  matrix
\begin{equation}\label{def:operator2'}
    \t{\oT}^{g}_{A,B} (x,y) = g(x+y) A(x) B(y) \,.
\end{equation}
Let us describe the simplest properties of matrices $\oT^{g}_{A,B}$ and $\t{\oT}^{g}_{A,B}$.
By Lemma \ref{l:singular_decomposition}, we have
$$
    \oT^{g}_{A,B} (x,y) = \sum_{j=0}^{|B|-1} \la_j (\oT^{g}_{A,B}) u_j (x) v_j (y)
$$
and similar for $\t{\oT}^{g}_{A,B}$.
Here $u_j,v_j$ are
singularfunctions.
We arrange the eigenvalues in order of magnitude
$$
    \la_0 (\oT^{g}_{A,B}) \ge \la_1 (\oT^{g}_{A,B}) \ge \dots \ge \la_{|B|-1} (\oT^{g}_{A,B}) \,,
$$
and similar for $\t{\oT}^{g}_{A,B}$.
We call $\la_0$ the main eigenvalue and $u_0$, $v_0$ the main singularfunctions.
Clearly,
\begin{equation}\label{f:TT*}
    \oT^{g}_{A,B} (\oT^{g}_{A,B})^* (y,y') = B(y) B(y') \Cf_3 (A,g,\ov{g}) (-y,-y') \,,
\end{equation}
\begin{equation}\label{f:TT*_tilde}
    \t{\oT}^{g}_{A,B} (\t{\oT}^{g}_{A,B})^* (y,y') = B(y) B(y') \Cf_3 (A,g,\ov{g}) (y,y') \,,
\end{equation}
\begin{equation}\label{f:T*T}
    (\oT^{g}_{A,B})^* \oT^{g}_{A,B} (x,x') = A(x) A(x') \Cf_3 (B,\ov{g}^c,g^c) (-x,-x') \,,
\end{equation}
\begin{equation}\label{f:T*T_tilde}
    (\t{\oT}^{g}_{A,B})^* \t{\oT}^{g}_{A,B} (x,x') = A(y) A(y') \Cf_3 (B,\ov{g},g) (x,x') \,.
\end{equation}
For real $g$, we get
$$
    (\t{\oT}^{g}_{A,B})^* = \t{\oT}^{g}_{B,A} \,,
$$
and for even real $g$, we obtain
$$
    (\oT^{g}_{A,B})^* = \oT^{g}_{B,A} \,.
$$
By Lemma \ref{l:singular_decomposition}, we have
\begin{equation}\label{f:sum_eigenvalues'}
    \sum_{j=0}^{|B|-1} \la^2_j (\oT^{g}_{A,B}) = \sum_{x,y} |g(x-y)|^2 A(x) B(y)
        \quad \mbox{ and } \quad
            \sum_{j=0}^{|B|-1} \la^2_j (\t{\oT}^{g}_{A,B}) = \sum_{x,y} |g(x+y)|^2 A(x) B(y) \,.
\end{equation}
Further
$$
    \sum_{j} \la_j^4 (\oT^{g}_{A,B})
        =  \sum_{y,y'} B(y) B(y') |\Cf_3 (A,g,\ov{g}) (-y,-y')|^2
            =
$$
\begin{equation}\label{f:sum_squares_eigenvalues'_1}
        = \sum_{x,x'} A(x) A(x') |\Cf_3 (B,\ov{g}^c,g^c) (-x,-x')|^2 \,,
\end{equation}
and
$$
    \sum_{j} \la_j^4 (\t{\oT}^{g}_{A,B})
                =
                    \sum_{y,y'} B(y) B(y') |\Cf_3 (A,g,\ov{g}) (y,y')|^2
                        =
$$
\begin{equation}\label{f:sum_squares_eigenvalues'_2}
    =
        \sum_{x,x'} A(x) A(x') |\Cf_3 (B,\ov{g},g) (x,x')|^2 \,.
\end{equation}

In the following lemma we find, in particular, all eigenvalues as well as all singularfunctions
of operators $\oT^{A-B}_{A,B}$, $\t{\oT}^{A+B}_{A,B}$.

\begin{lemma}
    Let $A,B\subseteq \Gr$ be finite sets, $|B| \le |A|$,
    $D,S\subseteq \Gr$ be two sets such that
    $A-B \subseteq D$, $A+B \subseteq S$.
    Then the main eigenvalues and singularsunctions
    of the operators $\oT^{D}_{A,B}$, $\t{\oT}^{S}_{A,B}$
    equal
    $\la_0 = (|A||B|)^{1/2}$,
    and
    $$
        v_0 (y) = B(y)/|B|^{1/2}\,,
            \quad \mbox{ and } \quad
        u_0 (x) = A(x)/|A|^{1/2}\,,
    $$
    correspondingly.
    All other singular values equal zero.
\label{l:eigenvalues_D,S'}
\end{lemma}
\begin{proof}
Using formulas  (\ref{f:TT*}), (\ref{f:TT*_tilde}) it is easy to see that
$$
    (\oT^{D}_{A,B} (\oT^{D}_{A,B})^* B) (y) = B(y) \sum_{y'\in B} \Cf_3 (A,D,D) (-y,-y')
    = |A| |B| B(y) \,,
$$
$$
    (\t{\oT}^{S}_{A,B} (\t{\oT}^{S}_{A,B})^* B) (y)
        = B(y) \sum_{y'\in B} \Cf_3 (A,S,S) (y,y') = |B| B(y) (A\c S) (y) = |A| |B| B(y) \,.
$$
Thus $v_0 (y) = B(y)/|B|^{1/2}$ and $\la_0 = (|A||B|)^{1/2}$.
It follows that
    $$
        u_0 (x) = A(x) (|A|^{1/2} |B|)^{-1} \sum_y B(y) D(x-y) = A(x) / |A|^{1/2} \,,
    $$
    and
    $$
        u_0 (x) = A(x) (|A|^{1/2} |B|)^{-1} (B\c S) (x)  = A(x) / |A|^{1/2} \,,
    $$
    correspondingly for $\oT^{D}_{A,B}$, $\t{\oT}^{S}_{A,B}$.
    By (\ref{f:sum_eigenvalues'}), we have
    $$
        \sum_{j=0}^{|B|-1} \la^2_j = |A| |B| \,.
    $$
    Thus all other singular values equal zero.
$\hfill\Box$
\end{proof}


\bigskip

Now we adapt the arguments from \cite{s_ineq}, see Proposition 28.

\begin{lemma}
    Let $A,B\subseteq \Gr$ be finite sets, $D,S\subseteq \Gr$ be two sets such that
    $A-B \subseteq D$, $A+B \subseteq S$.
    Suppose that $\psi$ be a
    function on $\Gr$.
    Then
\begin{equation}\label{f:3/2_energy_D'}
    |A|^2 \sigma^2 (\psi,B)
        \le
            \E_3(A,B) \sigma(\psi^2,D) \,,
\end{equation}
    and
\begin{equation}\label{f:3/2_energy_S'}
    |A|^2 \sigma^2 (\psi,B)
        \le
            \E_3(A,B) \sigma(\psi^2,S) \,.
\end{equation}
\label{l:3/2_energy'}
\end{lemma}
\begin{proof}
    Let us prove (\ref{f:3/2_energy_S'}),
    the proof of (\ref{f:3/2_energy_D'}) is similar.
    Denote by $\la_j$ the singular values of $\oT^{S}_{A,B}$
    and by $u_j$, $v_j$ the correspondent eigenfunctions.
    By Lemma \ref{l:eigenvalues_D,S'}, we, clearly, get
    $$
        S(x+y) A(x) B(y) = \sum_{j=0}^{|B|-1} \la_j u_j (x) v_j (y) = \la_0 u_0 (x) v_0 (y) \,.
    $$
    Using Lemma \ref{l:eigenvalues_D,S'} once more, we have
$$
    \sum_{x\in A}\, \sum_{y,z\in B} S(x+y) S(x+z) \psi (y-z)
        =
            \la^2_0 \sum_{y,z} \psi(y-z) v_0 (y) v_0 (z)
                =
                    |A| \sigma (\psi,B) \,.
$$
But
$$
     \sum_{x\in A}\, \sum_{y,z\in B} S(x+y) S(x+z) \psi (y-z)
                =
                    \sum_{\a,\beta} S(\a) S(\beta) \psi (\a-\beta) \Cf_3 (-A,B,B) (\a,\beta) \,.
$$
By Cauchy--Schwarz, we obtain
\begin{equation}\label{tmp:05.10.2012_1'}
    |A|^2 \sigma^2 (\psi,B)
        \le
            \E_3(A,B) \sum_{\a,\beta} S(\a) S(\beta) \psi^2 (\a-\beta)
                =
                    \E_3(A,B) \sigma(\psi^2,S)
\end{equation}
as required.
$\hfill\Box$
\end{proof}

\begin{corollary}
For any $A,B\subseteq \Gr$ the following holds
\begin{equation}\label{f:Li}
     |A|^2 \E^2_{3/2} (B)
        \le
            \E_3 (A,B) \E (B,A\pm B)
                \le
                    \E^{1/3}_3 (A) \E^{2/3}_3 (B) \E (B,A\pm B) \,.
\end{equation}
\end{corollary}

This inequality was obtained in \cite{Li}.

\begin{corollary}
For any $A\subseteq \Gr$ the following holds
\begin{equation*}\label{}
    |A|^6 \le \E_3 (A) \cdot \sum_{x\in A-A} ((A\pm A) \c (A\pm A)) (x) \,.
\end{equation*}
\end{corollary}

That is an inequality from \cite{ss2}.

\bigskip

Now we need in more symmetric version of the operators above.

Let $g : \Gr \to \C$ be a function, and $A\subseteq \Gr$ be a finite set.
By $\oT^{g}_A$ denote the matrix
\begin{equation}\label{def:operator1}
    \oT^{g}_A (x,y) = g(x-y) A(x) A(y) \,,
\end{equation}
and by $\t{\oT}^{g}_A (x,y)$ the matrix
\begin{equation}\label{def:operator2}
    \t{\oT}^{g}_A (x,y) = g(x+y) A(x) A(y) \,.
\end{equation}
General theory of such operators was developed in \cite{s},
and applications can be found in \cite{ss_E_k}, \cite{s}, \cite{s_ineq}, \cite{s_heilbronn}.
Here we describe the simplest properties of matrices $\oT^{g}_A$ and $\t{\oT}^{g}_A$.
It is easy to see that
$\oT^{g}_A$ is hermitian iff $\ov{g(-x)}=g(x)$
and
$\t{\oT}^{g}_A$ is hermitian iff $g$ is a real function.
Below we will deal with just hermitian operators with real functions $g$.
In the case we arrange the eigenvalues in order of magnitude
$$
    |\mu_0 (\oT^{g}_A)| \ge |\mu_1 (\oT^{g}_A)| \ge \dots \ge |\mu_{|A|-1} (\oT^{g}_A)| \,,
$$
and similar for $\t{\oT}^{g}_A$.
We call $\mu_0$ the main eigenvalue and the correspondent eigenfunction
as the main eigenfunction.
By Lemma \ref{l:singular_decomposition} the following holds
\begin{equation}\label{f:sum_eigenvalues}
    \sum_{j} \mu_j (\oT^{g}_A) = g(0) |A|
        \quad \mbox{ and } \quad
            \sum_{j} \mu_j (\t{\oT}^{g}_A) = \sum_x A(x) g(2x) \,.
\end{equation}
Further, in the case of hermitian (normal) $\oT^{g}_A$, $\t{\oT}^{g}_A$, we get
\begin{equation}\label{f:sum_squares_eigenvalues}
    \sum_{j} |\mu_j (\oT^{g}_A)|^2 = \sum_z |g(z)|^2 (A\c A) (z)
        \quad \mbox{ and } \quad
            \sum_{j} |\mu_j (\t{\oT}^{g}_A)|^2 = \sum_x |g(z)|^2 (A* A) (z) \,.
\end{equation}
Let also $f_0,f_1,\dots, f_{|A|-1}$ be the sequence of correspondent eigenfunctions.
Some results on the eigenfunctions can be found in \cite{s_ineq}.

\bigskip

Of course, the eigenvalues of operators $\oT^g_{A,A}$ and
the eigenvalues of operators $\oT^g_{A}$
are connected by a simple formula $\la_j (\oT^g_{A,A}) = |\mu_j (\oT^g_{A,A})|$,
provided by $\oT^g_{A}$ is a hermitian operator.
The same
formula holds for $\t{\oT}^g_{A,A}$, $\t{\oT}^g_{A}$.

\bigskip

\begin{example}
    One of the main ideas of using the operators of such sort is an attempt to find
    additively better subsets of $A$ than the whole set $A$.
    A typical example here is the following.
    Let $A=H\bigsqcup \L \subseteq \F_p^n$, where $H$ is a subspace and $\L$ is a dissociated set (basis).
    Suppose that $|H| \gg |A|^{2/3}$, $|H| \ll |A|$.
    Then $\E(A) \sim \E(H)$ and $A$ is not the main eigenfunction of the operator $\oT^{A\c A}_A$
    because of $\E(A)/|A| < \E(H)/|H| \le \E(A,H) / |H|$.
    Thus, 
    the main eigenfunction "sits" on $H$ not on whole $A$ in the case.
    Another idea
    of the operators method is an attempt to use "local"\, analysis on $A$
    in contrast to Fourier transformation method which is defined on the whole group $\Gr$.
\label{exm:H_cup_L}
\end{example}


We have an analog of Lemma \ref{l:eigenvalues_D,S'} with a similar proof.

\begin{lemma}
    Let $A\subseteq \Gr$ be a finite set, $D,S\subseteq \Gr$ be two sets such that
    $A-A \subseteq D$, $A+A \subseteq S$.
    Then $\oT^{D}_A$, $\t{\oT}^{S}_A$ have $\mu_0 = |A|$, $f_0 (x) = A(x)/|A|^{1/2}$
    and all other eigenvalues equal zero.
\label{l:eigenvalues_D,S}
\end{lemma}
\begin{proof}
    It is easy to see that $\mu_0 = |A|$ and $f_0 = A(x)/|A|^{1/2}$ in both cases.
    Further by formulas (\ref{f:sum_eigenvalues}), (\ref{f:sum_squares_eigenvalues}), we obtain
    $$
        \sum_j \mu_j = |A|
            \quad \mbox{ and } \quad
                \sum_j |\mu_j|^2 = |A|^2
    $$
    for eigenvalues of $\oT^{D}_A$ and $\t{\oT}^{S}_A$.
    Thus all other eigenvalues of both operators equal zero.
$\hfill\Box$
\end{proof}

\bigskip

\begin{remark}
    If we take $A=B$ in (\ref{f:3/2_energy_D'}), (\ref{f:3/2_energy_S'}), the function $\psi$ equals
    $\psi(x) = (A\c A) (x) / (D \c D)(x)$ or $\psi(x) = (A\c A) (x) / (S \c S)(x)$ then we get
    $$
        \sum_{x\in D} \frac{(A\c A)^2 (x)}{(D \c D)(x)} \le \frac{\E_3 (A)}{|A|^{2}} \,,
    $$
    and
    $$
        \sum_{x\in D} \frac{(A\c A)^2 (x)}{(S \c S)(x)} \le \frac{\E_3 (A)}{|A|^{2}} \,,
    $$
    A little bit sharper inequality of such form was obtained in \cite{s_ineq}.
\end{remark}

\bigskip

Besides formulas (\ref{f:sum_eigenvalues}), (\ref{f:sum_squares_eigenvalues})
there are some interesting relations between eigenvalues $\mu_\a (\oT^g_A)$
and eigenfunctions $f_\a$ of our operators.
By $g_\a$ denote the mean of the correspondent eigenfunction, that is
$g_\a = \sum_x f_\a (x)$.
We formulate our result for $\oT^g_A$.
Of course, for $\t{\oT}^g_A$ similar statement holds.

\begin{proposition}
    Let $g : \Gr \to \C$ be a function such that $\ov{g(-x)} = g(x)$.
    Then
\begin{equation}\label{f:mu_g_a_1}
    \sum_\a \mu_\a |g_\a|^2 = \sum_x g(x) (A\c A) (x) \,,
\end{equation}
\begin{equation}\label{f:mu_g_a_2}
    \sum_\a |\mu_\a|^2 |g_\a|^2 = \sum_{x\in A} | (g \c A) (x)|^2 \,.
\end{equation}
and if $g$ is a real even nonnegative function then
\begin{equation}\label{f:mu_g_a_3}
    \sum_\a \mu_\a |\mu_\a g_\a|^2 \ge \frac{1}{|A|^2} \left( \sum_x g(x) (A\c A) (x) \right)^3 \,.
\end{equation}
\label{p:mu_g_a}
\end{proposition}
\begin{proof}
Formula (\ref{f:mu_g_a_1}) follows from the definition of the operator $\oT^g_A$,
because of
$\langle \oT^g_A A, A \rangle = \sum_x g(x) (A\c A) (x)$.
To prove (\ref{f:mu_g_a_2}) note that
$$
    \mu_\a f_\a (x) = A(x) (g * f_\a) (x) \,.
$$
Thus
\begin{equation}\label{tmp:20.12.2012_1}
    \mu_\a g_\a = \sum_x A(x) (g * f_\a) (x) \,.
\end{equation}
Taking square of (\ref{tmp:20.12.2012_1}), summing over $\a$
and using orthogonality of $f_\a$, we get
$$
    \sum_\a |\mu_\a|^2 |g_\a|^2 = \sum_\a \sum_{x,x'\in A} f_\a (z) \ov{f_\a (z')} g(x-z) \ov{g(x'-z')}
        =
            \sum_{x,x'\in A}\, \sum_{z\in A} g(x-z) \ov{g(x'-z)}
                =
$$
$$
    = \sum_{x\in A} | (g \c A) (x)|^2 \,.
$$

To obtain (\ref{f:mu_g_a_3}) recall a useful inequality of A. Carbery \cite{Carbery_KB}
(see also \cite{Carbery_KB_elem}), namely,
\begin{equation}\label{tmp:20.12.2012_2}
    \langle \oT f_1, f_2 \rangle^3
        \le
            \|f_1 \|_3^3 \|f_2 \|_3^3
                \cdot
                    \sum_{x,y} \oT (x,y) \left( \sum_a \oT (x,a) \right) \left( \sum_b \oT (b,y) \right)
\end{equation}
which holds for $\oT, f_1,f_2 \ge 0$.
Because of
$$
    \oT^g_A (x,y) = \sum_\a \mu_\a f_\a (x) \ov{f_\a (y)}
$$
substitution $\oT = \oT^g_A$ and $f_1=f_2=A$
into (\ref{tmp:20.12.2012_2}) gives us
$$
    \left( \sum_x g(x) (A\c A) (x) \right)^3
        \le
            |A|^2 \cdot \sum_{x,y} \sum_\a \mu_\a f_\a (x) \ov{f_\a (y)}
                \left( \sum_\beta \ov{\mu_\beta f_\beta (x)} g_\beta  \right)
                    \left( \sum_\gamma \ov{\mu_\gamma g_\gamma} f_\gamma (y) \right)
    =
$$
$$
    = |A|^2 \cdot \sum_\a \mu_\a |\mu_\a g_\a|^2 \,.
$$
This completes the proof.
$\hfill\Box$
\end{proof}

\bigskip

Let $t$ be a positive integer.
By $(\oT^g_A)^{\otimes}$ denote the operator $\oT^{g^{\otimes}}_{A^{\otimes}}$,
where tensor power of the functions $g$ and $A$ is taken $t$ times.
We will call the obtained operator as $t$--tensor power of $\oT^g_A$.
Of course, if $\oT^g_A$ is hermitian then $(\oT^g_A)^{\otimes}$ is also hermitian.
Let us prove a result on tensor powers of operators $\oT^g_A$.

\begin{lemma}
    Let $t$ be a positive integer
    and $g : \Gr \to \C$ be a function such that $\ov{g(-x)} = g(x)$.
    Then eigenvalues and eigenfunctions of $t$--tensor power $(\oT^g_A)^{\otimes}$
    coincide with all $t$--products of eigenvalues and eigenfunctions of the operator $\oT^g_A$.
    In particular $\mu_0 ((\oT^g_A)^{\otimes}) = \mu^t_0 (\oT^g_A)$.
\label{l:tensor_operator}
\end{lemma}
\begin{proof}
Because of $\ov{g(-x)} = g(x)$ the operator $\oT^g_A$ is hermitian.
Let $\{ f_\a\}$, $\a\in [|A|]$ be the family of all orthonormal eigenfunctions of $\oT^g_A$.
Using formula (\ref{f:tensor_convolutions}) and the definition of
the operator $\oT^g_A$ it is easy to check that $|A|^t$
products of the form $\prod_{j=1}^{|A|} f_{\a_j}$, $\a_j \in \{ 0,1,\dots,|A|-1 \}$
are orthonormal eigenfunctions of $(\oT^g_A)^{\otimes}$.
This completes the proof.
$\hfill\Box$
\end{proof}

\bigskip

In terms of the eigenvalues it is very natural to formulate structural results
from papers \cite{ss_E_k}, \cite{s_ineq}.
Here we give for example a variant of Theorem 56 from \cite{s_ineq}.

\begin{theorem}
    Let $A\subseteq \Gr$ be a set, $\mu_0 = \oT^{A\c A}_A$, and $\E_3 (A) = M \mu_0^2$.
    Suppose that $M\le \mu_0 / (6|A|)$.
    Then there is a real number $r$
    \begin{equation}\label{cond:r_old}
        1\le r \le \frac{1}{|A|} \max_{x\neq 0} (A\c A)(x)
                    \cdot \frac{|A|^2}{\mu_0} M^{1/2}
            \le
                \frac{|A|^2}{\mu_0} M^{1/2} \,,
    \end{equation}
    and a set $A' \subseteq A$
    such that
    \begin{equation}\label{f:E_3_size_old}
        |A'| \gg  M^{-23/2} r^{-2} \log^{-9} |A| \cdot |A|  \,,
    \end{equation}
    and
    \begin{equation}\label{f:E_3_doubling_old}
        |nA'-mA'| \ll (M^9 \log^{6} |A|)^{7(n+m)} r^{-1} M^{1/2} \frac{|A|^2}{\mu_0} |A'|
    \end{equation}
    for every $n,m\in \N$.
\label{t:E_3_M_old}
\end{theorem}

\bigskip

Recall a lemma from \cite{s_ineq}.

\begin{lemma}
    Let $A\subseteq \Gr$ be a set, and $g$ be a nonnegative function,
    $\mu_0 = \mu_0 (\oT^{g}_A)$.
    Then
\begin{equation}\label{f:g_bound}
        |A| \ge \left( \sum_x f_0 (x) \right)^2
            \ge \max \left\{ \frac{\mu_0}{\| g \|_\infty} \,, \frac{\mu^2_0}{\| g \|_2^2} \right\} \,,
\end{equation}
    and
\begin{equation}\label{f:L_infty}
    \| f_0 \|_\infty \le \frac{\| g \|_2}{\mu_0} \,.
\end{equation}
If $\FF{g} \ge 0$ then
\begin{equation}\label{f:L_infty'}
    \| f_0 \|_\infty \le \frac{\| g_1 \|_2}{\mu^{1/2}_0} \,,
\end{equation}
where $g = g_1 \c \ov{g}_1$.
\label{l:g_bound}
\end{lemma}

Operator $\oT^{A\c A}_A$ is the simplest example of nonnegative defined operator on a set $A$.
On the other hand it is connected with the additive energy, because of
$|A|^{-1} \E(A) \le \mu_0 (\oT^{A\c A}_A)$ by Theorem \ref{t:Perron-Frobenius}, say.
Thus it is natural to try to obtain some estimates on the main eigenvalue of the operator.
We apply Lemma \ref{l:g_bound} to do this.
Another lower bounds for $\mu_0 (\oT^{A\c A}_A)$ are contained
in Theorem \ref{t:dual_bounds} of section \ref{sec:applications3}.

\begin{corollary}
    For any $A\subseteq \Gr$ the following holds
\begin{equation}\label{f:mu_energy_mu_g}
    \mu_0 (\oT^{A\c A}_A )
        \ge
            \max_{g\ge 0}\, \frac{\mu^3 (\oT^g_A)}{\|g\|_2^2 \cdot \|g\|_{\infty}} \,.
\end{equation}
\label{cor:mu_energy_mu_g}
\end{corollary}
\begin{proof}
Let $\mu =\mu_0 (\oT^g_A)$ and $f=f_0$ be the correspondent eigenfunction.
Instead of (\ref{f:mu_energy_mu_g}) we prove even stronger inequality,
namely, the same lower bound for $\langle \oT_A^{A\c A} f_0, f_0 \rangle$.
We have
$$
    \mu f (x) = A(x) (g * f) (x) \,.
$$
Thus
$$
    \mu^2 \left( \sum_x f(x) \right)^2
        \le
            \left( \sum_x g(x) (f\c A) (x) \right)^2
                \le
                    \| g\|_2^2 \E(A,f)
                        \le
                            \| g\|_2^2 \mu_0 (\oT^{A\c A}_A ) \,.
$$
Applying estimate (\ref{f:g_bound}) of Lemma \ref{l:g_bound},
we get
$$
    \left( \sum_x f(x) \right)^2 \ge \frac{\mu}{\| g\|_\infty}
$$
and the result follows.
$\hfill\Box$
\end{proof}

\bigskip



\bigskip

We conclude the section recalling a result from \cite{s_ineq}, Proposition 22
(or see the proof of Proposition \ref{p:mu_g_a}).

\begin{proposition}
    Let $A\subseteq \Gr$ be a set, $g_1,g_2$ be even real functions
    and $\{ f_\a \}$ be the eigenfunctions of the operator $\oT^{g_1}_A$.
    Then
$$
    \sum_{x,y,z\in A} g_1 (x-y) g_1 (x-z) g_2 (y-z)
        =
            \sum_{\a=0}^{|A|-1} \mu^2_\a (\oT^{g_1}_A) \cdot \langle \oT^{g_2}_A f_\a, f_\a\rangle \,.
$$
\label{p:triangles_g}
\end{proposition}

\section{Convex sets and sets with small multiplicative doubling}
\label{sec:applications}

Now we apply technique from section \ref{sec:eigenvalues} to obtain new
upper bounds for the additive energy of some families of sets.
Let us begin with the convex subsets of $\R$.

\begin{theorem}
    Let $A \subseteq \R$ be a convex set.
    Then
    \begin{equation}\label{f:convex_energy}
        \E (A) \ll |A|^{\frac{32}{13}} \log^{\frac{71}{65}} |A| \,.
    \end{equation}
\label{t:convex_energy}
\end{theorem}
\begin{proof}
Let $\E = \E(A) = |A|^3 / K$, $\E_3 = \E_3 (A)$, $L=\log |A|$.
Applying formula (\ref{f:E_3_gen_2-}) of Lemma \ref{l:E_3_convex} with $k=1$, we obtain
\begin{equation}\label{tmp:17.11.2012_1}
    2^{-2} \E \le \sum_{s ~:~ 2^{-1} |A| K^{-1} < |A_s| \le cK} |A_s|^2 \,,
\end{equation}
where $c>0$ is an absolute constant.
Put
$$
    D_j = \{ s\in A-A ~:~ 2^{j-2} |A| K^{-1} < |A_s| \le 2^{j-1} |A| K^{-1} \} \,,
$$
where $j \in [l]$, $2^l\le 2c^{} K^2 |A|^{-1} \ll K^2 |A|^{-1}$.
Thus by (\ref{tmp:17.11.2012_1}) the following holds
$$
    2^{-2} \E \le \sum_{j=1}^l \sum_{s\in D_j} |A_s|^2 \,.
$$
By pigeonhole principle, we find $j\in [l]$ such that
\begin{equation}\label{tmp:17.11.2012_D_pred}
    2^{-2} l^{-1} \E  \le \sum_{s\in D_j} |A_s|^2 \le |D_j| (2^{j-1} |A| K^{-1})^2 \,.
\end{equation}
Put $D=D_j$, $\Delta = 2^{j-1} |A| K^{-1}$, and $g(x) = (A\c A) (x) D(x)$.
Consider the operators $\oT_1 = \oT^g_{A}$,
$\oT_2 = \oT^{A}_{A,D}$ and $\oT_3 = \oT^{A\c A}_A$.
Clearly, all elements of matrices $\oT_1, (\oT_2)^* \oT_2$
does not exceed elements of
$\oT_3$ and
the operator $\oT_3$ is nonnegative  defined.
By formula (\ref{tmp:17.11.2012_D_pred}), we have
\begin{equation}\label{tmp:17.11.2012_D}
    \frac{\E}{4l|A|} \le \mu_0 (\oT_1) \,.
\end{equation}
Similarly,
\begin{equation}\label{tmp:17.11.2012_D'}
     \frac{\E}{4l|A|}
        \le
                \mu_0 (\oT_1)
                \le \langle \oT_3 f_0 , f_0 \rangle \,,
\end{equation}
where $f_0 \ge 0$ is the main eigenfunction of the operator $\oT_1$.
Applying Proposition \ref{p:triangles_g} with $A=A$, $g_1=g$, $g_2 = A\c A$,
we obtain
$$
    \mu^3_0 (\oT_1) \le \sum_{x,y,z\in A} g(x-y) g(x-z) (A\c A) (y-z)
$$
because of the operator $\oT_3$ is nonnegative  defined.
Further
\begin{equation}\label{tmp:17.11.2012_2}
    \mu^3_0 (\oT_1)
        \le
            \sum_{\a,\beta} g(\a) g(\beta) (A\c A) (\a-\beta) \Cf_3 (A) (\a,\beta) \,.
\end{equation}
The summation in (\ref{tmp:17.11.2012_2}) can be taken over $\a,\beta$ such that
$$
    (A\c A) (\a-\beta) \ge \frac{\E^2}{32 L^2 |A|^3 \E^{1/2}_3} := d \,.
$$
Indeed, otherwise by formulas (\ref{f:TT*}) and (\ref{f:sum_squares_eigenvalues'_1}), we get
$$
    \mu^3_0 (\oT_1)
        <
        d \Delta^2 \cdot \sum_{\a,\beta} D(\a) D(\beta) \Cf_3 (A) (\a,\beta)
            =
                d \Delta^2 \cdot \langle \oT_2 (\oT_2)^* D, D \rangle
                    \le
$$
$$
                    \le
                        d \Delta^2 |D| \mu_0 ( \oT_2 (\oT_2)^* )
                            \le
                        d \Delta^2 |D| \E^{1/2}_3
$$
and we obtain a contradiction in view of the definition of the set $D$
and inequality (\ref{tmp:17.11.2012_D}).
Thus
$$
    2^{-1} \mu^3_0 (\oT_1)
        \le
            \sum_{\a,\beta ~:~ (A\c A) (\a-\beta) \ge d} g(\a) g(\beta) (A\c A) (\a-\beta) \Cf_3 (A) (\a,\beta) \,.
$$
By Cauchy--Schwartz inequality and Lemma \ref{l:E_3_convex}, we get
$$
    \mu^6_0 (\oT_1)
        \ll
            \E_3
                \sum_{\a \in D,\, \beta \in D ~:~ (A\c A) (\a-\beta) \ge d}
                    (A\c A)^2 (\a) (A\c A)^2 (\beta) (A\c A)^2 (\a-\beta)
$$
\begin{equation}\label{tmp:17.11.2012_5}
    \ll
        |A|^3 L \D^3 \sum_{\a,\beta ~:~ (A\c A) (\a-\beta) \ge d} D(\a) (A\c A) (\a) D(\beta) (A\c A)^2 (\a-\beta)
            =
               |A|^3 L \D^3 \cdot \sigma \,.
\end{equation}
We estimate the quantity $\sigma$ in two different ways.
As above write
\begin{equation}\label{tmp:20.11.2012_1}
    S_i = \{ x ~:~ 2^{i-1} d < (A\c A) (x) \le 2^{i} d \} \,.
\end{equation}
Clearly, by Lemma \ref{l:E_3_convex}, we get $|S_i| \ll |A|^3 / (2^{i} d)^3$.
So for some $i$
\begin{equation}\label{tmp:17.11.2012_4}
    \sigma \ll L \sum_{\a,\beta, \a-\beta \in S_i} D(\a) (A\c A) (\a) D(\beta) (A\c A)^2 (\a-\beta)
        =
            L \sigma_*
\end{equation}
and your task is to estimate $\sigma_*$.
We put $\tau = 2^i d$ and write $S_\tau$ for $S_i$.
First of all by Lemma \ref{l:E_3_convex}, we have
\begin{equation}\label{tmp:17.11.2012_3}
    \sigma_* \ll \D \tau \E(D,A) \ll  \D \tau |A| |D|^{3/2} \,.
\end{equation}
Second of all, applying the same lemma twice and also the estimate $|S_\tau| \ll |A|^3 / \tau^3$, we obtain
\begin{equation}\label{tmp:17.11.2012_3'}
    \sigma_* \ll \tau^2 \sum_\a (A \c A) (\a) (D\c S_\tau) (\a)
        \ll
            \tau^2 |A| |D|^{3/4} |S_\tau|^{3/4}
                \ll
                    \tau^{-1/4} |A|^{13/4} |D|^{3/4} \,.
\end{equation}
Combining estimates (\ref{tmp:17.11.2012_3}), (\ref{tmp:17.11.2012_3'})
and optimizing over $\tau$, we derive
\begin{equation}\label{tmp:20.11.2012_2}
    \sigma_* \ll \Delta^{1/5} |A|^{14/5} |D|^{9/10} \,.
\end{equation}
Returning to (\ref{tmp:17.11.2012_5}), recalling (\ref{tmp:17.11.2012_D}), (\ref{tmp:17.11.2012_D'}),
and substituting the last formula into (\ref{tmp:17.11.2012_4}), we get
$$
    \frac{\E^6}{|A|^6 L^6}  \ll \mu^6_0 (\oT_1)
        \ll |A|^3 L^2 \D^3 \cdot \Delta^{1/5} |A|^{14/5} |D|^{9/10} \,.
$$
Accurate computations, using (\ref{tmp:17.11.2012_D_pred}) show
$$
    \left( \frac{\E}{|A| L} \right)^{51/10}
        \ll |A|^{29/5 + 9/10} L^2 \D^{7/5} \,.
$$
Applying estimate $\D \ll K$ after some calculations we obtain the result.
This completes the proof.
$\hfill\Box$
\end{proof}

\begin{corollary}
    Let $A\subseteq \Z$ be a convex set and
    $$
        P_A (\theta) = \sum_{a\in A} e^{2\pi i a \theta} \,.
    $$
    Then
    $$
        \int_0^{2\pi} |P_A (\theta)|^4\, d \theta \ll |A|^{\frac{32}{13}} \log^{\frac{71}{65}} |A| \,.
    $$
\end{corollary}

\begin{remark}
    The argument from the proof of Theorem \ref{t:convex_energy}
    is quite tight modulo our current knowledge of convex sets.
    Indeed, if one put $\tau=\D=K$ and, hence, by Lemma \ref{l:E_3_convex}
    we have $|D| \ll |A|^3 / K^3$
    then
    the estimate $K \gg |A|^{7/13 - }$ is obtained exactly.
    The same situation takes place in the case of multiplicative subgroups of $\Z/p\Z$, $p$ is a prime number
    (see \cite{s_ineq}),
    where
    the choice $\tau=\D=K$
    gives $K \gg |A|^{5/9 - }$.
\end{remark}

Probably, using similar arguments
one can obtain new upper bounds for $\T_k(A)$ as it was done in \cite{s_ineq}.
We do not make such calculations.
For $\T_k (A)$ weighted Szemer\'{e}di--Trotter theorem would
provide better bounds, probably.

\bigskip

Now we formulate a general result concerning the  additive energy of sets with small multiplicative doubling.

\begin{theorem}
    Let $A \subseteq \R$ be a set.
    Suppose that $|AA|=M|A|$, $M\ge 1$.
    Then
\begin{equation}\label{f:energy_gen}
    \E(A) \ll 
                            (M \log M)^{\frac{14}{13}}
                            |A|^{\frac{32}{13}} \log^{\frac{71}{65}} |A|
                \,.
\end{equation}
\label{t:energy_gen}
\end{theorem}
\begin{proof}
Let $\E = \E(A) = |A|^3 / K$, $\E_3 = \E_3 (A)$, $L=\log |A|$.
By Lemma \ref{l:arranging_gen}, we have $\E_3 (A) \ll (M \log M)^2 \cdot |A|^3 \log |A|$.
Thus $\E_3 (A)$ is small for small $M$ and we can apply the arguments from the proof of
Theorem \ref{t:convex_energy}.
Using
the first estimate of Lemma \ref{l:arranging_gen},
we obtain
$$
    2^{-2} \E \le \sum_{j=1}^l \sum_{s\in D_j} |A_s|^2 \,,
$$
where
$$
    D_j = \{ s\in A-A ~:~ 2^{j-2} |A| K^{-1} < |A_s| \le 2^{j-1} |A| K^{-1} \} \,,
$$
$c>0$ is an absolute constant and $j \in [l]$, $2^l \ll (M\log M)^2 K^2 |A|^{-1}$.
By pigeonhole principle, we find $j\in [l]$ such that
\begin{equation}\label{tmp:17.11.2012_D&}
    2^{-2} l^{-1} \E  \le \sum_{s\in D_j} |A_s|^2 \,.
\end{equation}
Put $D=D_j$, $\Delta = 2^{j-1} |A| K^{-1}$, and $g(x) = (A\c A) (x) D(x)$.
After that apply arguments in lines (\ref{tmp:17.11.2012_D_pred})---(\ref{tmp:17.11.2012_5}),
considering the operators $\oT_1 = \oT^g_{A}$,
$\oT_2 = \oT^{A}_{A,D}$, $\oT_3 = \oT^{A\c A}_A$,
and also the upper bound for $\E_3$, we have
$$
    \mu^6_0 (\oT_1)
        \ll
            \E_3
                \sum_{\a \in D,\, \beta \in D ~:~ (A\c A) (\a-\beta) \ge d}
                    (A\c A)^2 (\a) (A\c A)^2 (\beta) (A\c A)^2 (\a-\beta)
$$
$$
    \ll
        |A|^3 M^2 (\log M)^2 L \D^3 \sum_{\a,\beta ~:~ (A\c A) (\a-\beta) \ge d} D(\a) (A\c A) (\a) D(\beta) (A\c A)^2 (\a-\beta)
            =
$$
$$
            =
               |A|^3 M^2 (\log M)^2 L \D^3 \cdot \sigma \,.
$$
As is Theorem \ref{t:convex_energy} the number $d$ can be taken as $d = \frac{\E^2}{32 L^2 |A| \E^{1/2}_3}$.
Using a consequence of the first estimate of Lemma \ref{l:arranging_gen},
namely, $|S_i| \ll (M\log M)^2 |A|^3 /(d^3 2^{3i})$,
the second bound from Lemma \ref{l:arranging_gen},
and the arguments from lines (\ref{tmp:20.11.2012_1})---(\ref{tmp:20.11.2012_2}), we obtain
$$
    \sigma \ll \min_\tau \{ \D \tau |A| |D|^{3/2} (M \log M) , \tau^{-1/4} |A|^{13/4} |D|^{3/4} (M \log M)^{5/2} \}
$$
$$
    \ll
        \D^{1/4} |A|^{14/5} |D|^{9/10} (M \log M)^{11/5} \,.
$$
Thus
$$
    \frac{\E^6}{|A|^6 L^6}
        \ll |A|^3 L^2 \D^3 (M \log M)^2 \cdot \Delta^{1/5} |A|^{14/5} |D|^{9/10} (M \log M)^{11/5} \,.
$$
Accurate computations as is Theorem \ref{t:convex_energy} show
$$
    \left( \frac{\E}{|A| L} \right)^{51/10}
        \ll |A|^{29/5 + 9/10} L^2 \D^{7/5} (M \log M)^{21/5} \,.
$$
Using estimate $\D \ll K (M\log M)^2$ after some calculations, we obtain
$$
    K \gg |A|^{7/13} L^{-71/65} (M \log M)^{-14/13}
$$
as required.
$\hfill\Box$
\end{proof}

It is easy to check that Theorem \ref{t:energy_gen}
gives better bound for the additive energy then the bound from Lemma \ref{l:arranging_gen},
namely $\E(A) \ll M \log M |A|^{5/2}$ if, roughly, $M \ll |A|^{1/2-}$.

In \cite{s_ineq} the following theorem of the same type was obtained.

\begin{theorem}
    Let $A \subseteq \R$ be a set, and $\eps \in [0,1)$ be a real number.
    Suppose that $|AA|=M|A|$, $M\ge 1$, and
    \begin{equation}\label{f:S_eps&}
        |\{ x \neq 0 ~:~ (A\c A) (x) \ge |A|^{1-\eps} \}|
            \ll
                (M \log M)^{\frac{5}{3}} |A|^{\frac{1}{6} - \frac{\eps}{4}} \log^{\frac{5}{6}} |A| \,.
    \end{equation}
    Then
\begin{equation}\label{f:energy_gen&}
    \E(A) \ll 
                            M \log M
                            |A|^{\frac{5}{2} - \frac{\eps}{12}} \log^{\frac{1}{2}} |A|
                \,.
\end{equation}
\label{t:energy_gen&}
\end{theorem}

Thus, our Theorem \ref{t:energy_gen} is better than Theorem \ref{t:energy_gen&}
if, roughly speaking, $M\ll |A|^{\frac{1}{2} - \frac{13\eps}{12}}$.
The advantage of Theorem \ref{t:energy_gen} is the absence of $\eps$,
or, in other words, the absence more or less uniform upper bound for
the convolution of $A$, of course.

\bigskip

Apply arguments of the proof of Theorem \ref{t:energy_gen} for a new family of sets $A$ with small quantity $|A(A+1)|$.
Such sets were considered in \cite{A(A+1)}, where
the following lemma was proved.

\begin{lemma}
    Let $A,B\subseteq \R$ be two sets, and $\tau \le |A|,|B|$ be a parameter.
    Then
\begin{equation}\label{}
    | \{ s\in AB ~:~ |A\cap sB^{-1}| \ge \tau \} | \ll \frac{|A(A+1)|^2 |B|^2}{|A| \tau^3} \,.
\end{equation}
\label{l:arranging_product}
\end{lemma}

Lemma above implies that for any $A\subseteq \R$
the following holds
$\E^\m (A) \ll |A(A+1)| |A|^{3/2}$.
Also in \cite{A(A+1)} a series of interesting inequalities were obtained.
Here we formulate just one result.

\begin{theorem}
    Let $A\subseteq \R$ be a set.
    Then
    $$
        \E^\m (A,A(A+1)) \,, ~ \E^\m (A+1,A(A+1)) \ll |A(A+1)|^{5/2} \,.
    $$
\label{t:A(A+1)_old}
\end{theorem}

In \cite{s_ineq} all bounds of Theorem \ref{t:A(A+1)_old} were improved,
provided by an analog of inequality (\ref{f:S_eps&}) holds.
We prove the following result, having no such condition.

\begin{corollary}
    Let $A\subseteq \R$ be a set, $a\in \R$ be a number, $|A(A+1)| = M|A|$, $M\ge 1$.
    Then
    \begin{equation}\label{f:E^m_A(A+1)-}
        \E^\m (A,A+a) \ll
                        M^{\frac{14}{13}}
                            |A|^{\frac{32}{13}} \log^{\frac{71}{65}} |A|
    \end{equation}
    In particular
    \begin{equation}\label{f:E^m_A(A+1)}
        \E^\m (A) \ll
                        M^{\frac{14}{13}}
                            |A|^{\frac{32}{13}} \log^{\frac{71}{65}} |A|
    \end{equation}
\label{c:f:E^m_A(A+1)}
\end{corollary}
\begin{proof}
Put $A'=A+a$, and now the convolution is the cardinality of the set
$
    \{ a_1,a_2 \in A ~:~ x= a_1 a^{-1}_2 \} \,.
$
Lemma \ref{l:arranging_product} implies that $\E^\m_3 (A') \ll M^2 |A|^3 \log |A|$.
After that apply the arguments from the proof of Theorem \ref{t:energy_gen}.
$\hfill\Box$
\end{proof}

\section{Structural results}
\label{sec:applications2}


Previous results of section \ref{sec:applications}
say, basically, that if $\E_3 (A)$ is small
and
$A$ has some additional properties, which show that $A$ is "unstructured"\, in some sense
then
we can say something nontrivial about the additive energy of $A$.
Now we formulate (see Theorem \ref{t:E_3_M} below)
a variant of the principle using just smallness of $\E_3 (A)$
to show that $A$ has a structured subset.
There are several results of such type, see \cite{ss_E_k}, \cite{s_ineq}.
Our new theorem is the strongest one in the sense that its has minimal requirements.
From some point of view
these type of statements can be  called an optimal version of
Balog--Szemer\'{e}di--Gowers theorem, see \cite{ss_E_k}.

\begin{theorem}
    Let $A\subseteq \Gr$ be a set, $\E (A) = |A|^{3}/K^{}$,
    and $\E_3 (A) = M|A|^4 / K^2$.
    Then there is a set $A' \subseteq A$
    such that
    \begin{equation}\label{f:E_3_size}
        |A'| \gg  M^{-10} \log^{-15} M \cdot |A|  \,,
    \end{equation}
    and
    \begin{equation}\label{f:E_3_doubling}
        |nA'-mA'| \ll (M^{9} \log^{14} M)^{6(n+m)} K |A'|
    \end{equation}
    for every $n,m\in \N$.
    Moreover, if $s\in (1,3)$ is a real number,
    and we have the following condition $\E_s (A) = |A|^{s+1}/K^{s-1}$
    then for all $s \in (1,3/2]$
    there is a set $A' \subseteq A$
    such that
    \begin{equation}\label{f:E_3_size_s}
        |A'| \gg  M^{-(14-4s)/(3-s)} (s-1)^{21} \log^{-21} (M(s-1)^{-1}) \cdot |A|  \,,
    \end{equation}
    and
    \begin{equation}\label{f:E_3_doubling_s}
        |nA'-mA'| \ll (M^{5} (s-1)^{-20} \log^{20} (M(s-1)^{-1}))^{6(n+m)} K |A'| \,.
    \end{equation}
    Finally, if $s\in [3/2,3)$ then
    there is a set $A' \subseteq A$
    such that
    \begin{equation}\label{f:E_3_size_ss}
        |A'| \gg  M^{-(44-24s)/(3-s)} (3-s)^{21} \log^{-21} M \cdot |A|  \,,
    \end{equation}
    and
    \begin{equation}\label{f:E_3_doubling_ss}
        |nA'-mA'| \ll (M^{(45-25s)/(3-s)} (3-s)^{-20} \log^{20} M )^{6(n+m)} K |A'|
    \end{equation}
    for every $n,m\in \N$.
\label{t:E_3_M}
\end{theorem}
\begin{proof}
Let
$\E_s = \E_s (A) = |A|^{s+1} / K^{s-1}$,
$\E_3 = \E_3 (A)$, $L=2(3-s)^{-1} \log (4M (s-1)^{-1})$.
Because of $\E_3$ is small we can apply the arguments from the proof of Theorem \ref{t:energy_gen}.
Write
$$
    D_j = \{ x\in A-A ~:~ 2^{j-2} |A| K^{-1} < |A_x| \le 2^{j-1} |A| K^{-1} \} \,.
$$
Trivially
$$
    |D_j| (2^{j-2} |A| K^{-1})^3 \le \E_3 \,,
$$
and whence
\begin{equation}\label{tmp:28.07.2012_1*}
    |D_j| \ll \E_3 / (|A|^3 K^{-3} 2^{3j}) \,.
\end{equation}
Thus
$$
    (s-1) \E_s  \ll \sum_{j=1}^l \sum_{s} |A_s|^s \,,
$$
where $l$ can be estimated as $\log M^{1/(3-s)} = L$.
By pigeonhole principle we find $j\in [l]$ such that
\begin{equation}\label{tmp:17.11.2012_D&}
    (s-1) L^{-1} \E_s  \ll \sum_{s\in D_j} |A_s|^s \,.
\end{equation}
Put $D=D_j$, $\Delta = 2^{j-1} |A| K^{-1}$, and $g(x) = (A\c A)^{s-1} (x) D(x)$.
From (\ref{tmp:17.11.2012_D&}) it follows that
\begin{equation}\label{tmp:17.11.2012_D_and_tilde}
    |D| \gg \frac{(s-1) |A|^{} K^{}}{L M^{s/(3-s)}}
\end{equation}
and
\begin{equation}\label{tmp:17.11.2012_D_and_tilde'}
            \sum_{x\in D} (A\c A) (x) \gg \frac{(s-1) |A|^2}{L M^{(s-1)/(3-s)}} \,.
\end{equation}
After that apply arguments in lines (\ref{tmp:17.11.2012_D_pred})---(\ref{tmp:17.11.2012_5}),
considering the operators $\oT_1 = \oT^g_{A}$,
$\oT_2 = \oT^{A}_{A,D}$, $\oT_3 = \oT^{A\c A}_A$.
Using Corollary \ref{cor:mu_energy_mu_g}, we get
\begin{equation}\label{tmp:TILDE}
    \langle \oT_3 f_0, f_0 \rangle \ge \frac{\mu^3_0 (\oT_1)}{\| g\|^2_2 \| g \|_\infty}
        \gg
            \frac{|D|^2 \D^3}{|A|^3} := \sigma \,.
\end{equation}
In the case $s=2$ as in Theorems \ref{t:convex_energy}, \ref{t:energy_gen}, we have $\sigma \ge \mu_0 (\oT_1)$.
Further,
by Proposition \ref{p:triangles_g}, we obtain
$$
    (s-1)^6 \mu^4_0 (\oT_1) \sigma^2
        \ll
            \E_3
                \sum_{\a \in D,\beta \in D ~:~ (A\c A) (\a-\beta) \ge d}
                    (A\c A)^{2s-2} (\a) (A\c A)^{2s-2} (\beta) (A\c A)^2 (\a-\beta)
$$
\begin{equation}\label{tmp:20.11.2012_ev}
    \ll
            \E_3 \D^{4s-4} \sum_{x ~:~ (A\c A) (x) \ge d} (D\c D) (x) (A\c A)^2 (x) \,,
\end{equation}
where $d$ can be taken as
$d=\frac{(s-1)^3 \sigma \mu_0 (\oT_1)}{2^{13} L\D^{s-2} \E^{1/2}_3}$
(and $d=\frac{\mu^2_0 (\oT_1)}{32|A|\E^{1/2}_3}$ in the case $s=2$).
Applying Cauchy--Schwartz inequality, we have
$$
    \sum_x (A\c A)^2 (x) (D \c D) (x) \le \E^{2/3}_3 \left( \sum_{x} (D \c D)^3 (x) \right)^{1/3}
        \le
            \E^{2/3}_3 |D|^{1/3} \E^{1/3} (D) \,.
$$
Put $\E (D) = \mu |D|^3$.
Recalling (\ref{tmp:20.11.2012_ev}), we get
\begin{equation}\label{f:mu_sigma}
    (s-1)^6 \mu^4_0 (\oT_1) \sigma^2
        \ll
            \left( \frac{M |A|^4}{K^2} \right)^{5/3} \D^{4s-4} |D|^{4/3} \mu^{1/3} \,.
\end{equation}
We have $\D \ll M^{1/(3-s)} |A| / K$.
First of all consider the case $s=2$.
In the situation the following holds $\sigma \ge \mu_0 (\oT_1)$.
Thus,
an
accurate calculations give
$$
    \E (D) = \mu |D|^3 \gg \frac{|D|^3}{M^9 L^{14}} \,.
$$
By Balog--Szemer\'{e}di--Gowers Theorem \ref{t:BSzG}
there is
$D' \subseteq D$ such that
$|D'| \gg \mu |D|$
and
$
    |D'+D'| \ll \mu^{-6} |D'|
$.
Pl\"{u}nnecke--Ruzsa inequality (see e.g. \cite{tv}) yields
\begin{equation}\label{tmp:31.07.2012_1}
    |nD'-mD'| \ll \mu^{-6(n+m)} |D'| \,,
\end{equation}
for every $n,m \in \N$.
Using the definition of the set $D=D_j$ and inequality
(\ref{tmp:17.11.2012_D_and_tilde'}) (recall that we are considering the case $s=2$),
we find $x\in \Gr$ such that
\begin{equation}\label{tmp:31.07.2012_2}
    |(A-x) \cap D'| \gg \mu |A| L^{-1} M^{-1}
        \gg
            M^{-10} L^{-15} \cdot |A| \,.
\end{equation}
Put $A' = A\cap (D'+x)$.
Using (\ref{tmp:31.07.2012_1}), (\ref{tmp:31.07.2012_2}) and the definition of $\D$,
we obtain for all $n,m \in \N$
\begin{equation}\label{tmp:31.07.2012_2'''}
    |nA'-mA'| \le |nD'-mD'| \ll \mu^{-7(n+m)} |A| |A'| \D^{-1}
        \ll
            \mu^{-6(n+m)} K |A'|
\end{equation}
and the theorem is proved in the case $s=2$.

Now take any $s\in (1,3)$.
Returning to (\ref{f:mu_sigma}), using (\ref{tmp:TILDE}) and making similar computations,
we obtain
\begin{equation}\label{tmp:29.11.2012_1}
    \left( \frac{(s-1)\E_s}{L|A|} \right)^{20/3} \D^{10-20s/3} \ll |A|^{10/3} \mu^{1/3} \left( \frac{M |A|^4}{K^2} \right)^{5/3} \,.
\end{equation}
Suppose that $s\in (1,3/2]$.
In the case
$$
    \mu \gg \frac{(s-1)^{20}}{M^5 L^{20}}
$$
because of $\D \gg |A|/K$.
After that repeat the arguments above.
If $s\in [3/2,3)$ then (\ref{tmp:29.11.2012_1}) gives us
$$
      \left( \frac{(s-1) \E_s}{L|A|} \right)^{20/3} (M^{1/(3-s)} |A| K^{-1})^{10-20s/3} \ll |A|^{10/3} \mu^{1/3} \left( \frac{M |A|^4}{K^2} \right)^{5/3}
$$
because of $\D \ll M^{1/(3-s)} |A| / K$.
Computations show
$$
    \mu \gg \frac{1}{M^{(45-25s)/(3-s)} L^{20}} \,.
$$
After that repeat the arguments above once more.
This completes the proof.
$\hfill\Box$
\end{proof}

\bigskip

Certainly, inequality (\ref{f:E_3_size}) and the assumption $\E(A) = |A|^3/K$
of the Theorem \ref{t:E_3_M} imply that $|A'-A'| \gg_M K|A'|$.
Thus, we need in the multiple $K$ in (\ref{f:E_3_doubling_ss}).

Of course, using the definition of the number $\D$ more accurate one can improve estimates
(\ref{f:E_3_size}), (\ref{f:E_3_doubling}) a little bit.

\bigskip

\begin{remark}
For every convex set Theorem \ref{t:E_3_M} above easily gives a "nontrivial"\, estimate $\E(A) \ll |A|^{5/2-\eps_0}$,
where $\eps_0>0$ is an absolute constant.
Indeed, putting $M = \log |A|$,
using formula (\ref{f:convex_A'}) of Lemma \ref{l:E_3_convex} and the upper bound for the energy $\E_3 (A)$
follows from the lemma, we have for the set $A'$ from Theorem \ref{t:E_3_M} that
$$
    |A|^{7/4} \ll_M |A'+A'-A'| \ll_M |A|^4 \E^{-1} (A)
$$
and the result follows.
Applying more refine method from \cite{ss2} one can get even simpler proof.
Indeed, for so large $A'\subseteq A$ we have
$|A|^{3/2+\eps_1} \ll_M |A'-A'| \ll_M |A|^4 \E^{-1} (A)$,
$\eps_1>0$ is an absolute constant
and again we obtain a lower bound for $\eps_0$.
Interestingly, that lower bounds for doubling constants
give us upper bounds for the additive energy in the case.

The same proof takes place for multiplicative subgroups
$\G \subseteq \Z/p\Z$, where $p$ is a prime number
if one use Stepanov's method (see e.g. \cite{K_Tula} or \cite{sv})
or combine Stepanov's method with
recent
lower bounds for the doubling constant from \cite{ss,sv,s_ineq}.
Note that an estimate of the sort $\E(\G) \ll |\G|^{5/2-\eps}$ was known before
and was obtained by another variant of the eigenvalues method, see \cite{s_ineq}.
On the other hand
any multiplicative subgroup $\G$ of size $|\G| > p^\eps$
is an additive basis of $\Z/p\Z$ of order $C(\eps)$
(see e.g. \cite{Glibichuk_Konyagin}, \cite{Bourgain_more} and general sum--product inequalities in \cite{Rudnev_11}).
The fact that the same is true for sets with small multiplicative doubling
was obtained in \cite{Bourgain_DH}
(more precisely, Bourgain proved that Fourier coefficients of such sets
are small in average)
and this also implies that a "non--trivial"\, upper bound for $\E(\G)$ holds.
\label{r:worker-peasant}
\end{remark}

\bigskip

The arguments above allow replace the condition on $\E_3$
in Theorem \ref{t:E_3_M} onto the same condition on $\E_4$ easily.
By evenness the proof is simpler in the situation.
General result of the same type with another constants was obtained in \cite{ss_E_k}, see Theorem 54.
We include the proof here because of the next important Theorem \ref{t:E_4_T_4},
which can be obtained by almost the same arguments.

\begin{theorem}
    Let $s\in [8/5,4)$ be a real number,
    $A\subseteq \Gr$ be a set, $\E_s (A) = |A|^{s+1}/K^{s-1}$, and $\E_4 (A) = M |A|^5 / K^3$.
    Then there is a set $A' \subseteq A$
    such that
    \begin{equation}\label{f:E_4_size}
        |A'| \gg  M^{-(5s-5)/(4-s)} (4-s)^{6} \log^{-6} M \cdot |A|  \,,
    \end{equation}
    and
    \begin{equation}\label{f:E_4_doubling}
        |nA'-mA'| \ll (M^{(4s-4)/(4-s)} (4-s)^{-5} \log^{5} M)^{6(n+m)} K |A'|
    \end{equation}
    for every $n,m\in \N$.
    If $s\in (1,8/5]$ then
    there is a set $A' \subseteq A$
    such that
    \begin{equation}\label{f:E_4_size_small}
        |A'| \gg  M^{-3/(4-s)} (s-1)^6 \log^{-6} (M(s-1)^{-1}) \cdot |A|  \,,
    \end{equation}
    and
    \begin{equation}\label{f:E_4_doubling_small}
        |nA'-mA'| \ll (M (s-1)^{-5} \log^{5} (M(s-1)^{-1}) )^{6(n+m)} K |A'|
    \end{equation}
    for every $n,m\in \N$.
\label{t:E_4_M}
\end{theorem}
\begin{proof}
Let
$\E_4 = \E_4 (A)$, $L=2(4-s)^{-1} \log (4M(s-1)^{-1})$.
In terms of Theorem \ref{t:E_3_M}, we have
$$
    (s-1)^8 \mu^8_0 (\oT_1)
        \ll
            \left( \sum_{x,y,z,w \in A} g(x-y) g(y-z) g(z-w) g(w-x) \right)^2
                =
$$
$$
    =
        \left(  \sum_{\a,\beta,\gamma} \Cf_4 (A) (\a,\beta,\gamma)
                        g^2 (\a) g^2 (\beta-\a) g^2 (\gamma-\beta) g^2 (\gamma) \right)^2
            \ll
$$
$$
                \ll
                    \E_4 \cdot \sum_{\a,\beta,\gamma} g^2 (\a) g^2 (\beta-\a) g^2 (\gamma-\beta) g^2 (\gamma)
    \ll
        \E_4 \cdot \D^{8s-8} \E(D)
            \ll
                \frac{M|A|^5}{K^3} \cdot \D^{8s-8} \E(D) \,.
$$
Here $g(x) = D(x) (A\c A)^{s-1} (x)$, $\oT_1 = \oT^g_A$,
$D=D_j$, $\Delta = 2^{j-1} |A| K^{-1}$.
We have $2^j \ll M^{1/(4-s)}$ and hence $\D \ll M^{1/(4-s)} |A| K^{-1}$.
Note that the number $j$ can be estimated by $\log M^{1/(4-s)} = L$.
Put $\E(D) = \mu |D|^3$.
After some accurate  calculations, we get for $s\in [8/5, 4)$ that
$$
    \left( \frac{|A|^s}{K^{s-1} L} \right)^5
        \ll
    \left( (s-1) \frac{|A|^s}{K^{s-1} L} \right)^5 \ll \frac{M|A|^5}{K^3} \D^{5s-8} |A|^3 \mu
        \ll
            \frac{M|A|^5}{K^3} (M^{1/(4-s)} |A| K^{-1})^{5s-8} |A|^3 \mu \,.
$$
Whence
$$
    \mu \gg \frac{1}{M^{(4s-4)/(4-s)} L^5} \,.
$$
After that repeat the arguments from lines (\ref{tmp:31.07.2012_1})---(\ref{tmp:31.07.2012_2'''})
of the proof of Theorem \ref{t:E_3_M}.

If $s\in (1,8/5)$ then it is easy to see that
$$
    \mu \gg \frac{(s-1)^5}{ML^5} \,.
$$
Repeating the arguments above gives the result.
This concludes the proof.
$\hfill\Box$
\end{proof}

\bigskip

Of course Theorems \ref{t:E_3_M},  \ref{t:E_4_M} can be generalized onto higher moments
but such generalizations became weaker if one consider higher $\E_k$
because we should have deal with $\T_k (D)$
not
$\E(D)$.

Instead of this we take another characteristic of a set $A$, namely, its energy $\T_4 (A)$
and obtain structural theorem in the situation.
Similar
result
was obtained in \cite{ss_E_k}, see Theorem 60.
Theorem \ref{t:BK_structural} from the introduction
has the same form but weaker assumption.

\bigskip

\begin{theorem}
    Let $A\subseteq \Gr$ be a set, $\E_{3/2} (A) = |A|^{5/2}/K^{1/2}$, and $\T_4 (A) = M|A|^7 / K^3$.
    Then there is a set $A' \subseteq A$
    such that
    \begin{equation}\label{f:E_4_size}
        |A'| \gg  \frac{|A|}{M K}  \,,
    \end{equation}
    and
    \begin{equation}\label{f:E_4_energy}
        \E(A') \gg \frac{|A'|^3}{M} \,.
    \end{equation}
    If $s$ is a real number, $s\in (1,3/2]$ and
    we have the following condition
    $\E_{s} (A) = |A|^{s+1}/K^{s-1}$ then
    there is a set $A' \subseteq A$
    such that
    \begin{equation}\label{f:E_4_size_small}
        |A'| \gg  \frac{(s-1)^8 |A|}{M K \log^8 K}  \,,
    \end{equation}
    and
    \begin{equation}\label{f:E_4_energy_small}
        \E(A') \gg \frac{(s-1)^8 |A'|^3}{M \log^8 K} \,.
    \end{equation}
\label{t:E_4_T_4}
\end{theorem}
\begin{proof}
Let
$\E_4 = \E_4 (A)$, $\T_4 = \T_4 (A)$.
In terms of Theorems \ref{t:E_3_M}, \ref{t:E_4_M}, we have
$$
    \mu^8_0 (\oT_1)
        \ll
            \left( \sum_{x,y,z,w \in A} g(x-y) g(y-z) g(z-w) g(w-x) \right)^2
                \ll
$$
$$
                \ll
                    \E_4 \cdot \sum_{\a,\beta,\gamma} g^2 (\a) g^2 (\beta-\a) g^2 (\gamma-\beta) g^2 (\gamma)
    \ll
        \E_4 \cdot \T_4
            \ll
                \E_4 \cdot \frac{M|A|^7}{K^3} \,.
$$
Here
$g(x) = (A\c A)^{s-1} (x)$, $\oT_1 = \oT^g_A$.
Put $\E_4 = \mu |A|^5$.
First of all consider the case $s=3/2$.
After some calculations, we get
$$
    \mu \gg \frac{1}{M K} \,.
$$
Clearly,
$$
    \sum_{s~:~ |A_s| < 2^{-2}\mu |A|}\, \sum_t \E(A_s,A_t)
        \le \sum_{s~:~ |A_s| < 2^{-2}\mu |A|}\, \sum_t |A_s|^2 |A_t|
            < 2^{-2} \E_4 \,.
$$
Thus by Lemma \ref{l:E_k-identity} one has
\begin{equation}\label{tmp:26.11.2012_1}
    \sum_{s,t ~:~ |A_s|, |A_t| \ge 2^{-2}\mu |A|} \E(A_s,A_t) \ge 2^{-1} \E_4 \,.
\end{equation}
Put
$$
    \nu := \max_{|A_s|,\, |A_t| \ge 2^{-2}\mu |A|}\, \frac{\E(A_s,A_t)}{|A_s|^{3/2} |A_t|^{3/2}} \,.
$$
By (\ref{tmp:26.11.2012_1}), we have
$$
    2^{-1} \E_4 \le \nu \sum_{s,t} |A_s|^{3/2} |A_t|^{3/2} = \nu \E^2_{3/2} (A)
$$
and hence $\nu \ge 2^{-1} \mu K$.
It follows that there are $s,t$ such that $|A_s|, |A_t| \ge 2^{-2}\mu |A|$ and
$$
    \E(A_s,A_t) \ge \nu |A_s|^{3/2} |A_t|^{3/2} \gg  \frac{|A_s|^{3/2} |A_t|^{3/2}}{M}\,.
$$
Applying Cauchy--Schwartz inequality, we obtain the result.

Now let $s\in (1,3/2]$.
In the case we put $g(x) = D(x) (A\c A)^{s-1} (x)$, where the set $D$ is defined
as in Theorems \ref{t:E_3_M}, \ref{t:E_4_M}.
Then the following holds
$$
    (s-1)^8 \left( \frac{\E_s (A)}{|A| L} \right)^8
        \ll
            \E_4 \T_4 \Delta^{8s-12} \,,
$$
where $\D \gg |A|/K$ and $L \ll \log K$.
Hence
$$
    \mu \gg \frac{(s-1)^8}{M L^8 K} \,.
$$
After that repeat the arguments above.
This completes the proof.
$\hfill\Box$
\end{proof}

\bigskip

\begin{remark}
    All bounds of Theorem \ref{t:E_4_T_4} are tight as an example $A=H\dotplus \L \subseteq \F_2^n$ shows.
    Here $H \le \F_2^n$ as a subspace and $\L$ is a dissociated set (basis)
    (see also Example \ref{r:L+H_E_s} from section \ref{sec:applications3}).
    This set $A$ corresponds to the case $\a=0$ in Theorem \ref{t:BK_structural} from introduction.
    There are another more difficult examples which demonstrate  the same.
\label{r:L+H}
\end{remark}

\begin{remark}
    The proof of Theorem \ref{t:E_4_T_4} gives, in particular, that
$$
    \left( \frac{\E_{3/2} (A)}{|A|} \right)^{2k}
        \le
            \E_k (A) \T_k (A)
$$
and
\begin{equation}\label{f:E_k,T_k,sigma}
    \left( \frac{\sum_{x\in D} (A\c A) (x)}{|A|} \right)^{2k}
        \le
            \E_k (A) \T_{k/2} (D)
\end{equation}
for any sets $A,D\subseteq \Gr$ and even positive $k$.
Some particular case of the last formula appeared in \cite{ss_E_k}, see Lemma 3
and also Remark 61.
\end{remark}

\begin{remark}
    In Theorem \ref{t:E_4_T_4} we have found a set $A'$ with huge additive energy.
    Thus, by Balog--Szemer\'{e}di--Gowers Theorem there is a huge subset of $A'$ with small doubling,
    similarly to
    Theorem \ref{t:BK_structural} from introduction.
    Remark \ref{r:L+H} shows that there is an example demonstarating sharpness of our theorem
    and having parameter $\a=0$.
    It is easy to construct similar counterexample
    corresponding to the opposite case $\a=(1-\tau_0)/2$ in Theorem \ref{t:BK_structural}.
    Indeed,
    let $H_1,\dots,H_k$, where $k=[K^{1/2}]$ be some totally  disjoint subspaces of $\F_2^n$
    in the sense that
    $|H_1+\dots+H_k| = |H_1|\dots |H_k|$ (see e.g. \cite{Sanders_approximate_groups}).
    Put $A=\bigsqcup_{j=1}^k H_j$.
    Then $\T_t (A) \sim |A|^{2t-1}/K^{t-1}$, $\E_s (A) \sim |A|^{s+1}/K^{s/2}$ but there is no any nontrivial $X_j$ here.
    So, once more,
    in terms of $\E_4 (A)$ and $\T_4 (A)$
    our Theorem \ref{t:E_4_T_4} is the best possible (even for smaller $\E_{3/2} (A)$)
    as the
    example
    above
    shows.
    Of course if we know something on "height"\, (see \cite{BK_AP3,BK_struct} or section \ref{sec:applications3})
    of the set $A$ then such $X_j$ can appear
    (inequality (\ref{f:E_k,T_k,sigma}) cast light to this slightly).
    We discuss the example with subspaces $H_1,\dots,H_k$ in section \ref{sec:applications3}.
\label{r:self-dual}
\end{remark}

     The particular case when the parameter $s$ equals $1$ in the theorems of the section was considered in \cite{s_ineq},
     see also \cite{ss_E_k}.

\section{Dual popular sets}
\label{sec:applications3}


Let $k\ge 2$ be an integer and $c\in (0,1]$ be a real number.
Given a set $A\subseteq \Gr$,
we call a set $\mathcal{P} \subseteq \Gr^{k-1}$ a {\it $(k,c)$--dual} (or just dual) to a set $P \subseteq \Gr$
if
$$
    c\E^P_k (A)
        \le
    \sum_{x,y} P(x-y) A(x) A(y)
\m
$$
\begin{equation}\label{def:dual_sets}
\m
    \sum_{z_1,\dots,z_{k-1}} \mathcal{P} (z_1,\dots,z_{k-1})
        A(x+z_1) \dots A(x+z_{k-1}) A(y+z_1) \dots A(y+z_{k-1}) \,.
\end{equation}
In the same way we can define a {\it $(k,c)$--dual}  set $P \subseteq \Gr$ to a set $\mathcal{P} \subseteq \Gr^{k-1}$
if
$$
    c \E^{\mathcal{P}}_k (A)
        \le
    \sum_{x,y} P(x-y) A(x) A(y)
\m
$$
\begin{equation}\label{def:dual_sets'}
\m
    \sum_{z_1,\dots,z_{k-1}} \mathcal{P} (z_1,\dots,z_{k-1})
        A(x+z_1) \dots A(x+z_{k-1}) A(y+z_1) \dots A(y+z_{k-1}) \,.
\end{equation}
Of course a dual set is not unique
and we write the fact that $\mathcal{P}$ belongs to the family of dual sets of $P$ as $\mathcal{P} = P^*$
and vice versa.

It is easy to find a pair of dual sets, e.g. take for any $P$
the set $\mathcal{P} = A^{k-1} - \D_{k-1} (A)$
and for any $\mathcal{P}$ the set $P=A-A$.
Let us consider another examples.
Let $P \subseteq \Gr$ be a popular difference set in the sense that
$$
    P=\{ z ~:~ |A_z|^{k-1} \ge \E_k(A) (2|A|^2)^{-1} \} \,.
$$
Then
\begin{equation}\label{tmp:24.12.2012_2}
    2^{-1} \E_k (A) \le \E^P_k (A) = \sum_{z\in P} |A_z|^k
        = \sum_{z_1,\dots,z_{k}} A(z_1) \dots A(z_k) \Cf_{k+1} (P,A) (z_1,\dots,z_k) \,.
\end{equation}
If we put
$$
    \mathcal{P} = \{ (z_1,\dots,z_{k-1}) ~:~ \Cf_{k} (A) (z_1,\dots,z_{k-1}) \ge \E_k (A) (4|A|^k)^{-1} \}
$$
then
$$
    2^{-2} \E_k (A) \le 2^{-1} \E^P_k (A) \le  \sum_{z_1,,\dots,z_{k}} A(z_1) \dots A(z_k) \mathcal{P} (z_1-z_k,\dots, z_{k-1}-z_k)
        \Cf_{k+1} (P,A) (z_1,\dots,z_k)
            =
$$
$$
            =
                \sum_{x,y} P(x-y) A(x) A(y) \sum_{z_1,\dots,z_{k-1}} \mathcal{P} (z_1,\dots,z_{k-1})
        A(x+z_1) \dots A(x+z_{k-1}) A(y+z_1) \dots A(y+z_{k-1}) \,.
$$
Thus, we have constructed $(k,1/2)$--dual sets.
In the same way we can start from the inequality
$$
    2^{-1} \E_k (A) \le \E^{\mathcal{P}}_k = \sum_{(z_1,\dots,z_{k-1}) \in \mathcal{P}} \Cf_k^2 (A) (z_1,\dots,z_{k-1})
$$
and after that define $P$, showing that $P = \mathcal{P}^*$.
In these two examples $P$ and $\mathcal{P}$ are popular difference sets.
Thus, we call the pair $P,\mathcal{P}$ of the form as $(k,1/4)$--popular dual sets.
If we have
$$
    c \E_k (A)
        \le
    \sum_{x,y} P(x-y) A(x) A(y)
\m
$$
\begin{equation}\label{tmp:24.12.2012_star}
\m
    \sum_{z_1,\dots,z_{k-1}} \mathcal{P} (z_1,\dots,z_{k-1})
        A(x+z_1) \dots A(x+z_{k-1}) A(y+z_1) \dots A(y+z_{k-1})
\end{equation}
or, in other words, if $P,\mathcal{P}$ are {\it $(k,c)$--popular dual} sets
then, clearly,
$$
    c\E_k(A) \le \sum_{z\in P} |A_z|^k
        \quad \mbox{ and } \quad
            c\E_k (A) \le \sum_{(z_1,\dots,z_{k-1}) \in \mathcal{P}} \Cf_k^2 (A) (z_1,\dots,z_{k-1}) \,.
$$
So, there is a converse implication,
in some sense.

Note also that if $P$, $\mathcal{P}$ are $(k,c)$--popular dual sets
then by formulas (\ref{f:tensor_convolutions}), (\ref{f:tensor_convolutions_C_k})
for any integer $t$ tensor powers $P^\otimes$, $\mathcal{P}^\otimes$
form
$(k,c^t)$--popular dual sets for $A^\otimes$.
Another examples of popular dual sets are level popular difference sets, that is
$$
    P_i =\{ s ~:~ 2^{i-1} \E_k(A) (2|A|^2)^{-1} < |A_s|^{k-1} \le 2^{i} \E_k(A) (2|A|^2)^{-1} \} \,,
$$
and
$$
    \mathcal{P}_j = \{ (z_1,\dots,z_{k-1}) ~:~ 2^{j-1} \E_k (A) (4 |A|^k)^{-1} < \Cf_{k} (A) (z_1,\dots,z_{k-1}) \le 2^{j} \E_k (A) (4 |A|^k)^{-1} \}
$$
in the sense that there are $i,j\in [L]$ such that $P_i$ and $\mathcal{P}_j$
are $(k,2^{-2} L^{-2})$--popular dual sets, where $L = L(A) = \log (4 |A|^{k+1} \E^{-1}_k (A))$.
In the case we define
$$
    \D = \D(P) = \D(A,P) = 2^i \E_k(A) (2|A|^2)^{-1} \,,
$$
and
$$
    \D^* = \D(P^*) = \D(A,P^*) = 2^j \E_k (A) (4 |A|^k)^{-1} \,.
$$

\bigskip

The situation $k=2$ is the most interesting.
In the
case
there is a dual formula
\begin{equation}\label{f:dual_k=2}
    \sum_{x,y} A(x) A(y) g(x-y) \Cf_3 (h,A,A) (x,y)
        =
            \sum_{x,y} A(x) A(y) h(x-y) \Cf_3 (g,A,A) (x,y) \,,
\end{equation}
where $g,h : \Gr \to \C$ are any functions.
The formula above
shows that one has $(P^*)^* = P$ in the sense that the set $P$ is a popular dual from the family
of all popular dual sets of $P^*$.
We will write $P^*$ instead of $\mathcal{P}$ in the case
and $c$--popular dual instead of $(2,c)$--popular  dual.

\begin{example}
    Consider the set $A$ from Remark \ref{r:L+H}.
    The set of popular differences of $A$ naturally splits onto two sets
    $$
        P_1 = H = \{ x ~:~ |A_x| = |A| \}
        \quad
    \mbox{ and }
        \quad
        P_2 = \{ x \in (A-A) \setminus H ~:~ |A_x| = |H| \} \,.
    $$
    It is easy to check that $P_2 = P^*_1$ and vice versa.
    Note also that for $s\ge 1$ one has
    \begin{displaymath}
        \E_s (A) \sim |H| |A|^s + |A|^2 |H|^{s-1} \sim
            \left\{ \begin{array}{lll}
                    |H| |A|^s & \mbox{ if } & s\ge 2 \\
                    |A|^2 |H|^{s-1} & \mbox{ if } & s<2
            \end{array} \right.
    \end{displaymath}
    So, the sum over $P_1$ in $\E_s (A)$ dominates for large $s$
    and the sum over $P_2$ dominates for small $s$.
\label{r:L+H_E_s}
\end{example}

Now we prove the main result concerning properties of dual sets.
The most interesting part is inequality (\ref{f:k=2}), which
gives a non--trivial relation between $\E(A)$ and $\E_s (A)$, $s\in [1,2]$.
Bounds (\ref{f:D,D^*})--(\ref{f:s,s^*'}) and (\ref{f:k=2'})
say that there is a connection between some characteristics of dual sets.
As a consequence (see \cite{BK_AP3,BK_struct} or corollary and proposition  below), one can derive that
for any "regular"\, (see exact formulation below) set $A$  one can find a set $Q\subseteq A-A$
such that
$\E_Q (A) \gg \E^{1-} (A) := |A|^{3-} /K$ and $\sigma_Q (A) \gg |A|^{2-} / K^{1/2}$.

\bigskip

Consider the hermitian operator
\begin{equation}\label{def:T_op_P}
    \oT (x,y) = A(x) A(y) \sum_{z_1,\dots,z_{k-1}}
        \mathcal{P} (z_1,\dots,z_{k-1}) A(x+z_1) \dots A(x+z_{k-1}) A(y+z_1) \dots A(y+z_{k-1}) \,.
\end{equation}
It is easy to see that $\oT$ is nonnegative defined.
By formula (\ref{f:T*T})
for $k=2$ the operator coincide with $(\oT^A_{A,\mathcal{P}^c})^* \oT^A_{A,\mathcal{P}^c}$.

\begin{theorem}
Let $A\subseteq \Gr$ be a set.
Suppose that the notation above takes place.
Then there are two $(k,2^{-2} L^{-2})$--popular dual sets $P$, $\mathcal{P}$ such that
\begin{equation}\label{f:D,D^*}
    \D \D^* \le 16 L \mu_0 (\oT^{(A\c A)^{k-1}}_A) \,,
\end{equation}
and
\begin{equation}\label{f:s,s^*}
    \E^2_k (A) \le 16 L^2 \mu_0 (\oT^{(A\c A)^{k-1}}_A) \sigma_P (A) \sigma_{\mathcal{P}} (A) \,,
\end{equation}
\begin{equation}\label{f:P,P^*}
    \E^2_k (A) \le 2^{8} L^3 \mu^2_0 (\oT^{(A\c A)^{k-1}}_A) |P| |\mathcal{P}| \,,
\end{equation}
\begin{equation}\label{f:s,s^*'}
    \E^2_k (A)
        \le
            16 L^4 \sigma_P (A) \sigma_{\mathcal{P}} (A) \mu_0 (\oT) \,.
\end{equation}
In the case $k=2$ for any $s\in [1,2]$ the following holds
\begin{equation}\label{f:k=2}
    \E(A) \le
                    \mu_0 (\oT^{A\c A}_A)^{1-s/2} \E_s (A) 
\end{equation}
and if $P^*$ is $c$--dual to $P$ then
\begin{equation}\label{f:k=2'}
    c^2 \E^2_P (A) \le \sigma_{P^*} (A) \cdot \sum_{x,y\in P} \Cf^2_3 (A) (x,y) \,.
\end{equation}
\label{t:dual_bounds}
\end{theorem}
\begin{proof}
Firstly, we prove an approximate formulas up to some logarithms and constants
and after that we will use tensor trick (see e.g. \cite{tv}),
replacing $A$ by $A^t$, where $t$ is an integer.
By formula (\ref{f:tensor_convolutions}), we have $\E_s (A^t) = \E^t_s (A)$ for any $s$.
Applying Lemma \ref{l:tensor_operator}, we get
$\mu_0 (\oT^{(A\c A)^\otimes}_{A^\otimes}) = \mu^t_0 (\oT^{A\c A}_A)$.
Finally, $L(A^t) \le t L(A)$, where $L = \log (4 |A|^{k+1} \E^{-1}_k (A))$.

We begin with (\ref{f:D,D^*}).
Let $P=P_i$, $\mathcal{P} = \mathcal{P}_j$ are
$(k,2^{-2} L^{-2})$--popular dual sets, with $L$ defined above.
By the definition of the operator $\oT$ and formula (\ref{tmp:24.12.2012_star}), we have
\begin{equation}\label{tmp:14.12.2012_1}
    \E_k (A) 2^{-2} L^{-2}
        \le
            \sum_\a \mu_\a (\oT^P_A) \langle \oT f_\a , f_\a \rangle
                \le
\end{equation}
\begin{equation}\label{tmp:14.12.2012_2}
                \le
            \mu_0 (\oT^P_A) \sum_\a \langle \oT f_\a , f_\a \rangle =
                \mu_0 (\oT^P_A) \cdot \sum_{(z_1,\dots,z_{k-1}) \in \mathcal{P}} \Cf_k (A) (z_1,\dots,z_{k-1}) \,,
\end{equation}
where $\{ f_\a \}_{\a\in [|A|]}$ are the eigenfunctions of the operator $\oT^P_A$.
Here we have used the fact that $P$ is symmetric.
Multiplying
the last inequality
by $\D\D^*$ and using the definition of the sets $P$, $\mathcal{P}$, we get
$$
    \D \D^* \le 16 L \mu_0 (\oT^{P(A\c A)^{k-1}}_A) \le 16 L \mu_0 (\oT^{(A\c A)^{k-1}}_A) \,.
$$

Let us prove (\ref{f:s,s^*}), (\ref{f:P,P^*}) and (\ref{f:s,s^*'}).
Multiplying (\ref{tmp:14.12.2012_2}) by $\D \sigma_P (A)$, we get (\ref{f:s,s^*}).
Combining (\ref{f:D,D^*}) and (\ref{f:s,s^*}), we have (\ref{f:P,P^*}).
Returning to (\ref{tmp:14.12.2012_1}) and using Cauchy--Schwartz inequality, we obtain
$$
    \E^2_k (A) 2^{-4} L^{-4}
        \le
            \sum_\a \mu^2_\a (\oT^P_A) \cdot \sum_\a \langle \oT f_\a , f_\a \rangle^2 \,.
$$
Now applying formula (\ref{f:sum_squares_eigenvalues}) and Lemma \ref{l:convex_eigenvalues},
we have
$$
    \E^2_k (A) 2^{-4} L^{-4}
        \le
            \sigma_P (A) \sum_\a \mu^2_\a (\oT)
        \le
            \sigma_P (A) \mu_0 (\oT) \sum_\a \mu_\a (\oT)
                =
                    \sigma_P (A) \sigma_{\mathcal{P}} (A) \mu_0 (\oT)
$$
and (\ref{f:s,s^*'}) is proved.

In the case $k=2$, we have
$\t{\D} := \min \{ \D, \D^* \} \le (16 L \mu_0 (\oT^{A\c A}_A))^{1/2}$.
Suppose that the minimum in attained at $P^*$, the opposite case can be considered similarly.
Then
$$
    \E (A) \le 4 L^2 \sum_{x\in P^*} |A_x|^2 \le 4 L^2 (\t{\D})^{2-s} \E_s (A)
        \le
            4 L^2 (16L)^{1-s/2} \mu_0 (\oT^{A\c A}_A)^{1-s/2} \E_s (A) \,.
$$
Now using tensor trick, we obtain (\ref{f:k=2}).
Inequality (\ref{f:k=2'}) can be derived similarly.
This completes the proof.
$\hfill\Box$
\end{proof}

\bigskip

Example from Remark \ref{r:self-dual} shows that all inequalities of the theorem above are sharp.
We need in $\mu_0 (\oT^{A\c A}_A)$ in bound (\ref{f:k=2}) as
Example \ref{exm:H_cup_L} asserts us.


In Theorem \ref{t:dual_bounds} the quantity $\mu_0 (\oT^{(A\c A)^{k-1}}_A)$ or, more precisely,
quantity $\mu_0 (\oT^P_{A})$ for some popular set $P$ has appeared.
The next lemma shows that
one
can easily estimate
former
eigenvalue
for large subset of $A$.

\begin{lemma}
    Let $A\subseteq \Gr$ be a set.
    There is $A'\subseteq A$, $|A'| \ge |A|/2$
    such that $\mu_0 (\oT^P_{A'}) \le \frac{2\E(A)}{\Delta|A|}$
    for any set $P \subseteq \{ x ~:~ |A_x| \le \Delta \}$ and any real number $\Delta >0$.
    In particular, $\mu_0 (\oT^{A\c A}_{A'}) \le \frac{2\E(A)}{|A|}$.
\label{l:A'_0.5}
\end{lemma}
\begin{proof}
Let
$$
    A_1 = \{ x ~:~ ((A*A)\c A) (x) > 2\E(A)/|A| \} \,.
$$
It is easy to see that $|A_1| < |A|/2$.
Put $A' = A\setminus A_1$ and let $f$ be the main eigenfunction of the operator $\oT^P_{A'}$.
Let also $\mu_0 = \mu_0 (\oT^P_{A'})$.
We have
$$
    \mu_0 f(x) = A'(x) (P* f)(x) \,.
$$
Summing over $x\in A'$ and using the definition of the set $A'$, we obtain
$$
    \mu_0 \sum_x f(x) = \sum_x f(x) (P\c A') (x) \le \Delta^{-1} \sum_x f(x) ((A\c A) \c A ) (x)
        =
$$
$$
    =   \Delta^{-1} \sum_x f(x) ((A*A)\c A) (x)
        \le
            \Delta^{-1} \frac{2\E(A)}{|A|} \cdot \sum_x f(x)
$$
and we are done.
$\hfill\Box$
\end{proof}

\bigskip

We need in an analog of a definition from \cite{Sh_doubling}.

\begin{definition}
    Let $\a > 1$ be a real number, $\beta,\gamma \in [0,1]$.
    A set $A\subseteq \Gr$ is called $(\a,\beta,\gamma)$--connected
    if for any $B \subseteq A$, $|B| \ge \beta|A|$
    the following holds
    $$
        \E_\a (B) \ge \gamma \left( \frac{|B|}{|A|} \right)^{2\a} \E_\a (A) \,.
    $$
\end{definition}

Thus, a set from Theorem \ref{t:BK_structural} is a $(2,\beta,\gamma)$--connected set
with $\beta,\gamma \gg 1$.
As was proved in \cite{Sh_doubling} that for $\a=2$ every set $A$ always contains large connected subset.

We obtain a consequence of Theorem \ref{t:dual_bounds} for connected sets $A$.
Our inequality (\ref{f:connected}) below shows that there is a nontrivial connection
between $\E(A)$ and $\E_s (A)$, $1\le s \le 2$ in the case.

\begin{corollary}
    Let $A\subseteq \Gr$ be a set, and $\beta,\gamma \in [0,1]$.
    Suppose that $A$ is $(2,\beta,\gamma)$--connected with $\beta \le 1/2$.
    Then there are two $2^{-6} \gamma L^{-2}$--popular dual sets $P$, $P^*$ such that
\begin{equation}\label{f:c_dd}
    \D \D^* \le \frac{2^{8} L^2 \E (A)}{\gamma |A|} \,,
\end{equation}
\begin{equation}\label{f:c_pp}
    L^{-5} \gamma^3 2^{-21} |A|^2 \le |P| |P^*| \,,
\end{equation}
and
\begin{equation}\label{f:c_ss}
    L^{-3} \gamma^2 2^{-13} \E(A) |A| \le \sigma_P (A) \sigma_{P^*} (A) \,.
\end{equation}
    Further for any $s\in [1,2]$ the following holds
\begin{equation}\label{f:connected}
    \E_s (A) \ge 2^{-5} \gamma |A|^{1-s/2} \E^{s/2} (A) \,.
\end{equation}
\label{c:connected}
\end{corollary}
\begin{proof}
Let
$$
    P_j = \{ x ~:~ 2^{j-1} \gamma \E(A) /(2^5 |A|^2) < |A_x| \le 2^{j} \gamma \E(A) /(2^5 |A|^2) \} \,, \quad j\in [L] \,.
$$
Applying Lemma \ref{l:A'_0.5}, we find a set $A' \subseteq A$, $|A'| \ge |A|/2$
such that for any $j$ the following holds $\mu_0 (\oT^{P_j}_{A'}) \le 2 \E(A) / (\D_j |A|)$.
Here $\D_j = 2^{j} \gamma \E(A) /(2^5 |A|^2)$.
By connectedness, we have
\begin{equation*}\label{tmp:26.12.2012_2}
    \gamma 2^{-5} \E(A) \le 2^{-1} \E(A') \le \sum_{j=1}^L \sum_{x\in P_j} (A' \c A')^2 (x)
\end{equation*}
and for some $j\in [L]$ there is $P=P_j$ such that
\begin{equation}\label{tmp:26.12.2012_1}
    \gamma 2^{-5} L^{-1} \E(A) \le 2^{-1} L^{-1} \E(A') \le \sum_{x\in P} (A' \c A')^2 (x) = \E_P (A') \,.
\end{equation}
Of course
$$
    \E_P (A') \le \E_{P} (A) \le \D \sigma_P (A) \,.
$$
Consider $P^*$ and put $\D=\D_j$.
From (\ref{tmp:26.12.2012_1}) it follows that $P$, $P^*$ are
$2^{-6} \gamma L^{-2}$--popular dual sets.
Applying the arguments of Theorem \ref{t:dual_bounds} for second estimate from
(\ref{tmp:26.12.2012_1}), we obtain
\begin{equation}\label{tmp:26.12.2012_3}
    2^{-2} L^{-2} \E(A') \le \mu_0 (\oT_{A'}^P) \sigma_{P^*} (A')
        \le
            \frac{2\E(A)}{\D|A|} \cdot \sigma_{P^*} (A) \,.
\end{equation}
Multiplying the last inequality by $\D^*$ and using the connectedness again, we get
$$
    2^{-2} L^{-2} \D \D^* \E(A') \le \frac{2\E(A)}{|A|} 2 \E(A)
        \le
            2^{6} \gamma^{-1} \E(A') \frac{\E(A)}{|A|}
$$
and we obtain (\ref{f:c_dd}) for our $P$, $P^*$.

Further,
multiplying estimate (\ref{tmp:26.12.2012_3}) by $\sigma_P (A)$ and recalling (\ref{tmp:26.12.2012_1}), we have
$$
    \gamma 2^{-6} L^{-2} \E(A) \sigma_P (A) \D
        \le
            \frac{2\E(A)}{|A|} \cdot \sigma_{P^*} (A) \sigma_P (A) \,.
$$
By the definition of the set $P$ and inequality (\ref{tmp:26.12.2012_1}),
we get
$$
    \gamma 2^{-6} L^{-2} \E(A) \cdot \gamma 2^{-6} L^{-1} \E(A)
        \le
            \frac{2\E(A)}{|A|} \cdot \sigma_{P^*} (A) \sigma_P (A)
$$
and we obtain (\ref{f:c_ss}).
Combining (\ref{f:c_dd}) and (\ref{f:c_ss}), we get (\ref{f:c_pp}).

Finally,
applying inequality (\ref{f:k=2}) of Theorem \ref{t:dual_bounds},
we have
$$
    2^{-4} \gamma \E(A) \le \E(A') \le \mu^{1-s/2}_0 (\oT^{A'\c A'}_{A'}) \E_s (A')
        \le
                \mu^{1-s/2}_0 (\oT^{A\c A}_{A'}) \E_s (A)
            \le
                \left( \frac{2\E(A)}{|A|} \right)^{1-s/2} \E_s (A)
$$
and (\ref{f:connected}) follows.
$\hfill\Box$
\end{proof}

\bigskip

The example from Remark \ref{r:self-dual} shows that inequality (\ref{f:connected}) cannot be improved.
In the situation $P=P^* = \bigsqcup_{j=1}^k H_j$, so it is natural to call the set $A$ from the example as "self--dual"\, set.
Such sets have some interesting properties, for example, by corollary above they always
have relatively large $\sigma_P (A)$
and small $\D$.
The set from Remark \ref{r:L+H}, see also example \ref{r:L+H_E_s}, says that even
for connected $A$ estimate (\ref{f:connected}) does not satisfy for $s>2$.
So, in the region trivial estimates $\E_{s_2} (A) \le \E_{s_1} (A) |A|^{s_2-s_1}$, $s_2 \ge s_1$
can be sharp.
Finally, example \ref{exm:H_cup_L} shows that inequality (\ref{f:connected}) cannot hold for any $A$,
generally speaking, we need in connectedness of $A$,
or, similarly, we need in $\mu_0 (\oT^{A\c A}_A)$ not just $\E(A)/|A|$ to use inequality (\ref{f:k=2})
in the case of an arbitrary $A$.

Thus, in general, even for connected set $A$ we cannot find $P\subseteq A-A$
such that, roughly speaking,
$\E_P (A) \gg \E(A) = |A|^3 /K$ with $\sigma_P (A)$ greater than $|A|^2 / K^{1/2}$
and $\D \ll |A|/ K^{1/2}$.
Nevertheless if we
know lower bounds for
$\E_s (A)$, $s<2$ then
estimates
of Corollary \ref{c:connected}
can be improved.
We show this in the next two statements.

\begin{proposition}
    Let $A\subseteq \Gr$ be a set, $s\in (1,2]$ and $\beta,\gamma \in [0,1]$.
    Suppose that $A$ is $(s,\beta,\gamma)$--connected with $\beta \le 1/2$.
    Then there are two
    $$
        (2^{3s} \gamma^{-1} L^s)^{-1/(s-1)} \E^{1/(s-1)}_s (A) |A|^{-(4-2s)/(s-1)} \E^{-1} (A)
    $$
    --popular dual sets $P$, $P^*$ such that
\begin{equation}\label{f:c_dd+}
    \E_s (A) |A|^{s-1} \D (\D^*)^{s-1}
        \le
            2^{4s+3} \gamma^{-1} L^{s+1}
                \E^{s} (A) \,,
\end{equation}
and
\begin{equation}\label{f:c_ss+}
    \E^2_s (A) |A|^{s-1}
        \le
            2^{6s+1} \gamma^{-2} L^{s+1}
                \E^{s-1} (A)
                    \sigma^{s-1}_P (A) \sigma^{3-s}_{P^*} (A) \,.
\end{equation}
\label{p:E_s_E}
\end{proposition}
\begin{proof}
Let
$$
    P_j = \{ x ~:~ 2^{j-1} \gamma \E_s (A) / (2^{1+2s} |A|^2)
        < |A_x|^{s-1} \le 2^{j} \gamma \E_s (A) / (2^{1+2s} |A|^2) \} \,, \quad j\in [L] \,.
$$
Applying Lemma \ref{l:A'_0.5}, we find a set $A' \subseteq A$, $|A'| \ge |A|/2$
such that for any $j$ the following holds $\mu_0 (\oT^{P_j}_{A'}) \le 2 \E(A) / (\D_j |A|)$.
Here $\D^{s-1}_j = 2^{j} \gamma \E_s (A) / (2^{1+2s} |A|^2)$.
By connectedness, we have
\begin{equation}\label{tmp:26.12.2012_2'}
    \gamma 2^{-1-2s} \E_s (A) \le 2^{-1} \E_s (A') \le \sum_{j=1}^L \sum_{x\in P_j} (A' \c A')^s (x)
\end{equation}
and for some $j\in [L]$ there is $P=P_j$ such that
\begin{equation}\label{tmp:26.12.2012_1'}
    \gamma 2^{-1-2s} L^{-1} \E_s (A) \le 2^{-1} L^{-1} \E_s (A') \le \sum_{x\in P} (A' \c A')^s (x)
        = \E^{P}_s (A') \,.
\end{equation}
Of course
$$
    \E^{P}_s (A') \le \E^{P}_s (A) \le \D^{s-1} \sigma_P (A) \,.
$$
By H\"{o}lder's inequality, we have
\begin{equation}\label{tmp:16.12.2012_0}
    \E^{P}_s (A') \le \E^{s-1}_P (A') \sigma^{2-s}_P (A')
        \,.
\end{equation}
Now let $P^*$ be a dual to the set $P$.
In particular
\begin{equation}\label{tmp:24.12.2012_0}
    \E_P (A') \le 2L \sum_\a \mu_\a (\oT^P_{A'}) \langle \oT f_\a , f_\a \rangle
        \le
            \frac{4L\E(A)}{\D|A|} \sigma_{P^*} (A) \,,
\end{equation}
where $\{ f_\a \}$ are the eigenfunctions of the operator $\oT^P_{A'}$
and $\oT$ is the operator defined by (\ref{def:T_op_P}) with $\mathcal{P} = P^*$.
From (\ref{tmp:26.12.2012_1'}), (\ref{tmp:16.12.2012_0}) and
a trivial upper bound $\sigma_P (A) \le |A|^2$
it follows that $P$ and $P^*$
are
$$
        (2^{3s} \gamma^{-1} L^s)^{-1/(s-1)} \E^{1/(s-1)}_s (A) |A|^{-(4-2s)/(s-1)} \E^{-1} (A)
$$
--popular dual sets.
Put $\sigma = \sigma_P (A)$ and $\sigma^* = \sigma_{P^*} (A)$.
Substitution (\ref{tmp:24.12.2012_0}) into (\ref{tmp:16.12.2012_0}) gives us
in view of inequality (\ref{tmp:26.12.2012_1'}) that
\begin{equation}\label{tmp:24.12.2012_1}
    \E_s (A) \D^{s-1}
        \le
            2^{4s-1} \gamma^{-1} L^s
                \left( \frac{\E(A)}{|A|} \right)^{s-1} (\sigma^*)^{s-1} \sigma^{2-s} \,.
\end{equation}
Multiplying by $\sigma$,
and using (\ref{tmp:26.12.2012_2'}),
we obtain
$$
    \E^2_s (A)
        \le
            2^{6s+1} \gamma^{-2} L^{s+1}
                \left( \frac{\E(A)}{|A|} \right)^{s-1}
                    (\sigma^*)^{s-1} \sigma^{3-s}
$$
and (\ref{f:c_ss+}) is proved.

Similarly, multiplying inequality (\ref{tmp:24.12.2012_1}) by $(\Delta^*)^{s-1} (\Delta)^{2-s}$
and using estimates
$$
    \D \sigma \le 2 \E(A) \,,\quad \quad \D^* \sigma^* \le 2 \E(A) \,,
$$
we have
$$
    \E_s (A) |A|^{s-1} \D (\D^*)^{s-1}
        \le
            2^{4s+3} \gamma^{-1} L^{s+1}
                \E^{s} (A) \,.
$$
This completes the proof.
$\hfill\Box$
\end{proof}

\bigskip

For example, if $\E(A) = |A|^3/K$, $\E_{3/2} (A) = |A|^{5/2} / K^{1/2}$
then from (\ref{f:c_dd+}) it follows that $\min \{ \D, \D^*\} \ll_L |A| / K^{2/3}$
instead of $\min \{ \D, \D^*\} \ll_L |A| / K^{1/2}$
which
is a consequence of
Theorem \ref{t:dual_bounds} provided by
we have an appropriate bound for $\mu_0 (\oT^{A\c A}_A)$
or Corollary \ref{c:connected}.

\begin{corollary}
    Let $A\subseteq \Gr$ be a set, $|A-A|\le K|A|$, and $s\in (1,2]$ be a real number.
    Then there are two
    \begin{equation}\label{f:pop_dual_c}
        (c L^{s})^{1/(s-1)} \frac{|A|^3}{K \E(A)}
    \end{equation}
    --popular dual sets $P$, $P^*$, where $c>0$ is an absolute constant,
     such that
\begin{equation}\label{f:c_dd+'}
    \D (\D^*)^{s-1} \ll L^{s+1} \cdot \frac{K^{s-1} \E^s (A)}{|A|^{2s}} \,,
\end{equation}
and
\begin{equation}\label{f:c_ss+'}
    |A|^{3s+1} \ll K^{2(s-1)} L^{s+1} \E^{s-1}(A) \sigma^{s-1}_P (A) \sigma^{3-s}_{P^*} (A) \,.
\end{equation}
\label{l:A-A_E}
\end{corollary}
\begin{proof}
Let
$$
    P_j = \{ x ~:~ 2^{j-1} |A|^{s-1} / (2^{2s+1} K^{s-1}) < |A_x|^{s-1} \le 2^{j} |A|^{s-1} / (2^{2s+1} K^{s-1}) \} \,, \quad j\in [L] \,.
$$
Applying Lemma \ref{l:A'_0.5}, we find a set $A' \subseteq A$, $|A'| \ge |A|/2$
such that for any $j$ the following holds
$\mu_0 (\oT^{P_j}_{A'}) \le 2 \E(A) / (\D_j |A|)$.
Here $\D^{s-1}_j = 2^{j} |A|^{s-1} / (2^{2s+1} K^{s-1})$.
Application of H\"{o}lder inequality gives
$$
    \E_s (A') \ge 2^{-2s} |A|^{s+1} / K^{s-1} \,.
$$
Then for some $j\in [L]$ and $P_j = P$ the following holds
$$
    \E_s (A') \le 2L \sum_{x\in P} |A'_x|^s \,.
$$
After that apply the arguments of Proposition \ref{p:E_s_E}.
This concludes the proof.
$\hfill\Box$
\end{proof}


\bigskip

Now we
try to
prove an analog of Theorem \ref{t:E_4_T_4} with a weaker assumption on the set $A$,
namely, the largeness of the additive energy.
More precisely, we obtain a lower bound for $\E_4 (A)$, and the existence of structural subset $A' \subseteq A$
follows similarly as in Theorem \ref{t:E_4_T_4}.
Something can be proved using dual technique, but we add one more optimization in the argument.
Our result is a very simple, the reason why we cannot get
bounds
similar to Theorem \ref{t:E_4_T_4}
is discussed after Proposition \ref{p:E_4_T_4_E}.

\begin{proposition}
    Let $A\subseteq \Gr$ be a set, $\E (A) = |A|^{3}/K^{}$,
    and $\T_4 (A) = M|A|^7 / K^3$.
    Then
    \begin{equation}\label{f:E_4_E}
        \E_4 (A) \ge \frac{|A|^5}{2^{5} L^{10/3} M^{1/3} K^{7/3}}  \,.
    \end{equation}
\label{p:E_4_T_4_E}
\end{proposition}
\begin{proof}
Let $P$ be a popular difference set, such that
\begin{equation}\label{tmp:24.12.2012_3}
    \E(A) \le 2L \sum_{x\in P} |A_x|^2 \,.
\end{equation}
From (\ref{tmp:24.12.2012_3}) we, clearly, have
\begin{equation}\label{tmp:24.12.2012_4}
    \E_4 (A) \ge (8L)^{-1} K^{-1} \D^2 |A|^3 \,.
\end{equation}
On the other hand from the arguments of Theorem \ref{t:E_4_T_4} it follows that
\begin{equation}\label{tmp:24.12.2012_5}
    \left( \frac{\E (A)}{2L|A|} \right)^8
        \le
            \E_4(A) \D^4 \T_4 (A)
                \le
                    \E_4(A) \D^4 \frac{M|A|^7}{K^3}  \,.
\end{equation}
Combining (\ref{tmp:24.12.2012_4}), (\ref{tmp:24.12.2012_5})
and making an optimization over $\D$, we obtain the result.
This completes the proof.
$\hfill\Box$
\end{proof}

\bigskip

The example from the Remark \ref{r:self-dual} shows that $K^{7/3}$ in (\ref{f:E_4_E}) cannot be replaced by
something smaller than $K^2$.
The reason why we have exactly $K^2$ in the example is clear.
Indeed, it is easy to see that there are $[K^{1/2}]$ eigenvalues equals, roughly, $\mu_0 (\oT^{A\c A}_A)$
in the case and our approximation of the sum $\sum_\a \mu^4_\a (\oT^{A\c A}_A)$
by just zero term is
unappropriate.
Quick calculations show that using all $[K^{1/2}]$ eigenvalues in the sum $\sum_\a \mu^4_\a (\oT^{A\c A}_A)$,
we
get
$K^2$.
It is interesting to obtain better estimate than (\ref{f:E_4_E}).

The same example shows that the way of finding structured $A'\subseteq A$
obtaining lower bounds for $\E_4 (A)$ is unappropriate in general.
Indeed, in the proof of Theorem \ref{t:E_4_T_4} we use Lemma \ref{l:E_k-identity}
which asserts that $\E_4 (A) = \sum_{s,t} \E (A_s,A_t)$.
But in the example for typical $(s,t)$ the energy $\E (A_s,A_t)$ is
pretty small and the average arguments give almost nothing.

\bigskip

\begin{remark}
    Of course the assumption of Theorem \ref{t:E_4_T_4}
    is stronger than the condition of Proposition \ref{p:E_4_T_4_E}.
    Such type of assumptions, namely,
    lower bounds for
    $\E_s (A)$, $s\le 2$
    can be interpreted as closeness of $A$ to be a set with small doubling.
    Indeed, by H\"{o}lder inequality all such conditions are included to each other
    and if $A$ has small doubling then all $\E_s (A)$, $s\ge 1$ are large.
\end{remark}

\bigskip

We conclude the section considering one more example of dual sets.
Let
$$
    P = \{ x ~:~ \E(A,A_x) \ge |A_x|^2 \E_3 (A) (2\E(A))^{-1} \} \,.
$$
Then by Lemma \ref{l:E_k-identity}
$$
    2^{-1} \E_3 (A) \le \sum_{x\in P} \E(A,A_x) = \sum_{x,y} A(x) A(y) (A\c A) (x-y) \Cf_3 (P,A,A) (x,y) \,.
$$
The last expression looks similar to (\ref{tmp:24.12.2012_2}).
As above define $P^*$ by formula
$$
    P^* = \{ x ~:~ |A_x| \ge \E_3 (A) (4\E(A))^{-1} \} \,.
$$
Thus
\begin{equation}\label{f:dual_E_3}
    2^{-2} \E_3 (A) \le \sum_{x,y} A(x) A(y) (A\c A) (x-y) P^* (x-y) \Cf_3 (P,A,A) (x,y) \,.
\end{equation}
Applying Cauchy--Schwartz inequality to
estimate
(\ref{f:dual_E_3}), we get
$$
    2^{-4} \E^2_3 (A) \le \E^{P^*}_3 (A) \cdot \sum_{x,y\in P} \Cf^2_3 (A) (x,y) \,.
$$
In particular,
$$
    \sum_{x\in P^*} |A_x|^3 \gg \E_3 (A)
$$
for a dual set $P^*$.
Using (\ref{f:dual_E_3}) one can obtain an analog of Theorem \ref{t:dual_bounds}
for such type of dual sets, of course.

\bigskip

\no{Division of Algebra and Number Theory,\\ Steklov Mathematical
Institute,\\
ul. Gubkina, 8, Moscow, Russia, 119991}
\\
and
\\
Delone Laboratory of Discrete and Computational Geometry,\\
Yaroslavl State University,\\
Sovetskaya str. 14, Yaroslavl, Russia, 150000
\\
and
\\
IITP RAS,  \\
Bolshoy Karetny per. 19, Moscow, Russia, 127994\\
{\tt ilya.shkredov@gmail.com}

\end{document}